\documentclass[11pt,reqno]{amsart}
\usepackage{graphicx,amsmath,amsthm,verbatim,amssymb,lineno,titletoc}
\usepackage{bbm}
\usepackage{hyperref}
\hypersetup{colorlinks}
\usepackage[numbers]{natbib}
\usepackage{array}

\newcolumntype{M}[1]{>{\centering\arraybackslash}m{#1}}
\newcolumntype{N}{@{}m{0pt}@{}}

\pdfoutput=1

\newtheorem{theorem}{Theorem}[section]
\newtheorem{prop}[theorem]{Proposition}

\newtheorem{lemma}[theorem]{Lemma}
\newtheorem{corollary}[theorem]{Corollary}

\theoremstyle{remark}
\newtheorem{remark}[theorem]{Remark}
\newtheorem*{example}{Example}

\numberwithin{counter}{section}

\theoremstyle{definition}
\newtheorem{definition}[theorem]{Definition}

\def\Re{\mathrm{Re}}
\def\one{\mathbf{1}}

\def\N{\mathbb{N}}
\def\Z{\mathbb{Z}}

\def\R{\mathbb{R}}
\def\EE{\mathbb{E}}

\def\zero{\mathbf{0}}
\newcommand\old[1]{}

\newcommand{\TV}{\mathrm{TV}}

\newcommand{\tmix}{t_\mathrm{mix}}
\newcommand{\trel}{t_\mathrm{rel}}
\newcommand{\dmin}{d_\mathrm{min}}
\newcommand{\dmax}{d_\mathrm{max}}

\newcommand{\e}{\varepsilon}

\newcommand{\C}{\mathbb{C}}
\newcommand{\G}{\mathcal{G}}
\newcommand{\HH}{\mathcal{H}}
\newcommand{\T}{\mathbb{T}}
\newcommand{\transpose}{\mathsf{T}}

\newcommand{\V}{\widetilde{V}}
\newcommand{\W}{\mathbb{W}}
\DeclareMathOperator{\Var}{\textrm{Var}}

\DeclareMathOperator{\Hom}{\textrm{Hom}}

\DeclareMathOperator{\Span}{\textrm{Span}}
\DeclareMathOperator{\interior}{\textrm{int}}

\newcommand{\st}{\,:\,}

\newcommand{\overbar}[1]{\mkern 2.5mu\overline{\mkern-2.5mu#1\mkern-2.5mu}\mkern 2.5mu}

\allowdisplaybreaks

\begin{document}

\title[Mixing time and eigenvalues of the abelian sandpile]{Mixing time and eigenvalues of the abelian sandpile Markov chain}

\author{Daniel C. Jerison}
\address{Daniel C. Jerison, Department of Mathematics, Cornell University, Ithaca, NY 14853.}
\email{\texttt{jerison@math.cornell.edu}}

\author{Lionel Levine}
\address{Lionel Levine, Department of Mathematics, Cornell University, Ithaca, NY 14853.}
\urladdr{\url{http://www.math.cornell.edu/~levine}}

\author{John Pike}
\address{John Pike, Department of Mathematics, Cornell University, Ithaca, NY 14853.}
\email{\texttt{jpike@cornell.edu}}
\urladdr{\url{http://www.math.cornell.edu/~jp999}}

\thanks{D.C. Jerison and J. Pike were supported in part by NSF grant DMS-0739164.}
\thanks{L. Levine was supported by NSF grant \href{http://www.nsf.gov/awardsearch/showAward?AWD_ID=1455272}{DMS-1455272} and a Sloan Fellowship.}

\date{November 2, 2015}


\begin{abstract}
The abelian sandpile model defines a Markov chain whose states are integer-valued functions on the vertices of a simple connected graph $G$. By viewing this chain as a (nonreversible) random walk on an abelian group, we give a formula for its eigenvalues and eigenvectors in terms of `multiplicative harmonic functions' on the vertices of $G$. We show that the spectral gap of the sandpile chain is within a constant factor of the length of the shortest non-integer vector in the dual Laplacian lattice, while the mixing time is at most a constant times the smoothing parameter of the Laplacian lattice. We find a surprising inverse relationship between the spectral gap of the sandpile chain and that of simple random walk on $G$: If the latter has a sufficiently large spectral gap, then the former has a \emph{small} gap! In the case where $G$ is the complete graph on $n$ vertices, we show 
that the sandpile chain exhibits cutoff at time $\frac{1}{4\pi^{2}}n^{3}\log n$.
\end{abstract}

\maketitle

\section{Introduction}
\label{Introduction}

Let $G=(V,E)$ be a simple connected graph with $n$ vertices, one of which is designated the 
\emph{sink} $s \in V$.
A \emph{sandpile} on $G$ is a function
	\[ \sigma : V \setminus \{s\} \to \N \]
from the nonsink vertices to the nonnegative integers. In the abelian sandpile model \cite{BTW88,Dha}, certain sandpiles are designated as \emph{stable}, and any sandpile can be \emph{stabilized} by a sequence of local moves called \emph{topplings}. (For the precise definitions see Section~\ref{The Sandpile Chain}.)
Associated to the pair $(G,s)$ is a Markov chain whose states are the stable sandpiles. To advance the chain one time step, we choose a vertex $v$ uniformly at random, increase $\sigma(v)$ by one, and stabilize. We think of $\sigma(v)$ as the number of sand grains at vertex $v$. Increasing it corresponds to dropping a single grain of sand on the pile. Topplings redistribute sand, and the role of the sink is to collect extra sand that falls off the pile.

This Markov chain $(\sigma_t)_{t \in \N}$ can be viewed as a 
random walk on a finite abelian group, the \emph{sandpile group}, and consequently its stationary distribution is uniform on the recurrent states.  How long does it take to get close to this uniform distribution? We seek to answer this question in terms of graph-theoretic properties of $G$ and algebraic properties of its Laplacian. 
What graph properties control the mixing time? In other words, how much randomly added sand is enough to make the sandpile `forget' its initial state? 

As we will see in Section~\ref{The Sandpile Chain}, the characters of the sandpile group are indexed by \emph{multiplicative harmonic functions} on $G$. These are functions
	\[  h: V \to \C^* \]
satisfying $h(s)=1$ and
	\[ h(v)^{\deg(v)} = \prod_{w \sim v} h(w) \]
for all $v \in V$, where $\deg(v)$ is the number of edges incident to $v$, and we write $w \sim v$ if $\{v,w\} \in E$. There are finitely many such functions, and they form an abelian group whose order is the number of spanning trees of $G$. Associated to each such $h$ is an eigenvalue of the sandpile chain,
\begin{equation}
\label{lambda-h}
\lambda_h = \frac{1}{n} \sum_{v \in V} h(v).
\end{equation}

Our first result relates the spectral gap of the sandpile chain to the shortest vector in a lattice. Both gap and lattice come in two flavors: continuous time and discrete time. Writing $\HH$ for the set of multiplicative harmonic functions, the continuous time gap is defined as
\[
\gamma_c = \min\{ 1 - \Re(\lambda_h) \,:\, h \in \HH, \, h \not\equiv 1 \}
\]
and the discrete time gap is defined as
\[
\gamma_d = \min\{ 1 - |\lambda_h| \,:\, h \in \HH, \, h \not\equiv 1 \}.
\]

At times it will be useful to enumerate the vertices as $V = \{v_1,\ldots,v_n\}$. In this situation we always take $v_n$ to be the sink vertex. The \emph{full Laplacian} of $G$ is the $n \times n$ matrix
\begin{align}
\label{e.thelaplacian}
\overbar{\Delta}(i,j)= \begin{cases}
\deg(v_i), & i=j\\
-1, & v_i \sim v_j\\
0, & \textrm{else.} \end{cases} 
\end{align}
The \emph{reduced Laplacian}, which we will denote by $\Delta$, is the submatrix of $\overbar{\Delta}$ omitting the last row and column (corresponding to the sink). Note that $\Delta$ is invertible, but $\overbar{\Delta}$ is not, since its rows sum to zero.

The lattice that controls the continuous time gap is $\Delta^{-1} \Z^{n-1}$, the integer span of the columns of $\Delta^{-1}$. To define the analogue for the discrete time gap, we consider the \emph{Moore-Penrose pseudoinverse} $\overbar{\Delta}^+$ of the full Laplacian. We also define:
\begin{align*}
\R_0^n: & \text{ the subspace of vectors in $\R^n$ whose coordinates sum to zero,} \\
\Z_0^n: & \text{ the set of vectors in $\R_0^n$ with integer coordinates,} \\
\W^n: & \text{ the orthogonal projection of $\Z^n$ onto $\R_0^n$.}
\end{align*}

\begin{theorem}
\label{t.gap.intro}
The spectral gap of the continuous time sandpile chain satisfies
	\[ 8\frac{\|y\|_2^2}{n} \leq \gamma_c \leq 2\pi^2 \frac{\|y\|_2^2}{n}, \]
where $y$ is a vector of minimal Euclidean length in $(\Delta^{-1}\Z^{n-1}) \setminus \Z^{n-1}$. 

The spectral gap of the discrete time sandpile chain satisfies
	\[ 8\frac{\|z\|_2^2}{n} \leq \gamma_d \leq 2 \pi^2 \frac{\|z\|_2^2}{n}, \]
where $z$ is a vector of minimal Euclidean length in $(\overbar{\Delta}^+ \Z_0^n) \setminus \W^n$.
\end{theorem}

The upper and lower bounds in Theorem~\ref{t.gap.intro} match up to a constant factor. In Theorem~\ref{lazy_gap_bd} below, we will also prove a simple lower bound for the discrete time gap, namely
	\begin{equation} \label{e.d2n} \gamma_d \geq \frac{8}{d_*^2 n} \end{equation}
where $d_*$ is the penultimate term in the nondecreasing degree sequence. The reason for the appearance of the second-largest degree is that $\gamma_d$ (unlike $\gamma_c$) does not depend on the choice of sink, which can be moved to the vertex of largest degree.  

The standard bound of the mixing time 
in terms of spectral gap is $\tmix = O(\gamma_d^{-1} \log |\mathcal{G}|)$ where $\mathcal{G}$ is the sandpile group. The size of this group---the number of spanning trees of $G$---can be exponential in $n \log d_*$, so the factor of $\log | \mathcal{G} |$ can be significant. The resulting upper bound on the mixing time obtained from \eqref{e.d2n} is often far from the truth. To improve it, our next result gives an upper bound for the $L^2$ mixing time of the sandpile chain in terms of a lattice invariant called the \emph{smoothing parameter}.
Let $\Lambda$ be a lattice in $\R^m$, and let $V = \Span(\Lambda) \subseteq \R^m$. Denote the dual lattice by $\Lambda^\ast = \{x \in V \st \langle x,y \rangle \in \Z \text{ for all } y \in \Lambda \}$. For $s > 0$, the function
\[
f_{\Lambda}(s) = \sum_{x \in \Lambda^\ast \setminus \{\zero\}} e^{-\pi s^2 \|x\|_2^2}
\]
is continuous and strictly decreasing, with a limit of $\infty$ as $s \to 0$ and a limit of $0$ as $s \to \infty$. For $\e > 0$, the smoothing parameter of $\Lambda$ is defined as $\eta_\e(\Lambda) := f_\Lambda^{-1}(\e)$.  

\begin{theorem}
 \label{t.smoothing.intro}
Fix $\e > 0$ and let $\sigma$ be any recurrent state of the sandpile chain. Let $H_\sigma^t$ and 
$P_\sigma^t$ denote the distributions at time $t$ of the continuous and discrete time sandpile chains, respectively, 
started from $\sigma$. The $L^2$ distances from $U$, the uniform distribution on recurrent 
states, satisfy:
\begin{align*}
\|H_{\sigma}^t - U \|_2^2 \leq \e 
\enspace & \text{\em for all }
t \geq \frac{\pi}{16} n \cdot \eta_\e^2(\Delta \Z^{n-1}), \\[5pt]
\|P_{\sigma}^t - U \|_2^2 \leq \e 
\enspace & \text{\em for all } 
t \geq \frac{\pi}{16}n \cdot \eta_\e^2(\overbar{\Delta} \Z^n).
\end{align*}
\end{theorem}

The smoothing parameter $\eta_\e$ can be bounded in terms 
of $n$ and $d_\ast$, allowing us to show in Theorem \ref{smoothing-mixing-time} that every sandpile chain has mixing time of order at most $d_\ast^2 n \log n$. In the 
case of the complete graph $K_n$, we use eigenfunctions to provide a matching lower bound. The upper and lower bounds taken together demonstrate that the sandpile chain on $K_n$ exhibits total variation cutoff.

\begin{theorem}
\label{t.cutoff.intro}
Let $P_{\sigma}^{t}$ be the distribution of the discrete time sandpile chain on $K_{n}$ after $t$ steps 
started at a fixed recurrent state $\sigma$, and let $U$ be the uniform distribution on recurrent states. 
For any $c \geq 5/4$,
\begin{align*}
\| P_{\sigma}^{t}-U \|_{\TV} \leq e^{-c} 
\enspace & \text{\em for all }
t\geq\frac{1}{4\pi^{2}}n^{3}\log n + cn^{3}, \\[5pt]
\| P_{\sigma}^{t}-U\| _{\TV} \geq 1-e^{-35c}
\enspace & \text{\em for all } 
0 \leq t\leq\frac{1}{4\pi^{2}}n^{3}\log n - cn^{3}.
\end{align*}
\end{theorem}

While the lower bound is specific to $K_n$, the upper bound holds for \emph{any} graph with $n$ vertices. Therefore, as the underlying graph varies with a fixed number of vertices, the sandpile chain on the complete graph mixes asymptotically slowest: For any graph sequence $G_n$ such that $G_n$ has $n$ vertices,
	\[ \limsup_{n \to \infty} \frac{\tmix(G_n)}{\tmix(K_n)} \leq 1. \]

\subsection{An inverse relationship for spectral gaps}

Given the central role of the graph Laplacian $\overbar{\Delta}$ in Theorems~\ref{t.gap.intro} and~\ref{t.smoothing.intro}, one might ask
how the eigenvalues $\{ \lambda_h \,:\, h \in \mathcal{H}\}$ of the sandpile chain relate to the eigenvalues of the Laplacian itself.  In particular, how does mixing of the sandpile chain on $G$ relate to mixing of the simple random walk on $G$?  The sandpile chain typically has an exponentially larger state space, so there is not necessarily a simple relationship.  Indeed, certain local features of $G$ that have little effect on the Laplacian spectrum can decrease the spectral gap of the sandpile chain. An example is the existence of two vertices $x,y$ with a common neighborhood, which enables the multiplicative harmonic function
	 \[ h(x) = e^{2\pi i / d}, \quad h(y) = e^{-2\pi i/d}, \quad h(z) = 1 \text{ for }z \neq x,y, \]
where $d = \deg(x) = \deg(y)$. The corresponding eigenvalue $\lambda_h$ is close to $1$, so the spectral gap of the sandpile chain is rather small whenever two such vertices exist.

Nevertheless there is a curious inverse relationship between the spectral gap of the sandpile chain and the spectral gap of the Laplacian.  According to our next result, if the simple random walk on $G$ mixes sufficiently quickly, then the sandpile chain on $G$ mixes \emph{slowly}!

\begin{theorem} 
\label{t.inverse.intro}
The spectral gap of the discrete time sandpile chain satisfies
\[
\gamma_d \leq \frac{4 \pi^2}{\beta_1^2 n}
\]
where $0 = \beta_0 < \beta_1 \leq \cdots \leq \beta_{n-1}$ are the eigenvalues of the full Laplacian $\overbar{\Delta}$. 
\end{theorem}

\noindent For a bounded degree expander graph, this upper bound on $\gamma_d$ matches the lower bound \eqref{e.d2n} up to a constant factor.

\subsection{Related work}

We were drawn to the topic of sandpile mixing times in part by its intrinsic mathematical interest and in part by questions arising in statistical physics. Many questions about the sandpile model, even if they appear unrelated to mixing, seem to lead inevitably to the study of its mixing time. For instance, the failure of the `density conjecture' \cite{FLW,JJ,PPPR} and the distinction between the stationary and threshold states \cite{L15} are consequences of slow mixing.  Our characterization of the eigenvalues of the sandpile chain may also help explain the recent finding of `memory on two time scales' for the sandpile dynamics on the $2$-dimensional square grid \cite{SMKW15}.

The significance of the characters of $\mathcal{G}$ for the sandpile model was first remarked by Dhar, Ruelle, Sen and Verma \cite{DRSV95}, who in particular analyzed the sandpile group of the square grid graph with sink at the boundary. Our multiplicative harmonic functions $h$ are related to the toppling invariants $Q$ of \cite{DRSV95} by $h = e^{2\pi i Q}$.

In the combinatorics literature the abelian sandpile model is known as `chip-firing' \cite{BLS91}. 
For an alternative construction of the sandpile group, using flows on the edges instead of functions on the vertices, see \cite{BHN97}; also \cite{B99} and \cite[Ch.\ 14]{GR01}. 
The geometry of the Laplacian lattice $\overbar \Delta \Z^n$ is studied in \cite{AM10,PPW11}
in connection with the combinatorial Riemann-Roch theorem of Baker and Norine \cite{BN07}. We do not know of a direct connection between that theorem and sandpile mixing times, but the central role of the Laplacian lattice in both is suggestive.

The pseudoinverse $\overbar{\Delta}^+$ has appeared before in the context of sandpiles: It is used by Bj\"{o}rner, Lov\'{a}sz and Shor \cite{BLS91} to bound the number of topplings until a configuration stabilizes; see \cite{G14} for a recent improvement. In addition, the pseudoinverse is a crucial ingredient in the `energy pairing' of Baker and Shokrieh~\cite{BS13}.

The smoothing parameter was introduced by Micciancio and Regev \cite{MR07} in the context of lattice-based cryptography. Our interest lies in results that relate the smoothing parameter to other lattice invariants, many of which have natural interpretations in the setting of the sandpile chain. Relationships of this kind have been found by \cite{P08, GPV08, WTW14}.

\subsection{Outline}

After formally defining the sandpile chain and recalling how it can be expressed as a random walk on the sandpile group, Section \ref{The Sandpile Chain} applies the classical eigenvalue formulas and mixing bounds for random walk on a group, giving the characterization \eqref{lambda-h} of eigenvalues in terms of multiplicative harmonic functions. Although the sandpile chain is nonreversible in general, the usual mixing bounds hold thanks to the orthogonality of the characters of the sandpile group. Our first application of \eqref{lambda-h} is the lower bound \eqref{e.d2n} on the spectral gap $\gamma_d$. Section \ref{The Sandpile Chain} concludes by showing how one can sometimes exploit local features of $G$, which we call `gadgets,' to infer an upper bound on $\gamma_d$. In many cases, this upper bound matches the lower bound from \eqref{e.d2n} up to a constant factor.

The next two sections are the heart of the paper. Section \ref{Dual Lattices} proves a correspondence between multiplicative harmonic functions and equivalence classes of vectors in the dual Laplacian lattices $\Delta^{-1} \Z^{n-1}$ and $\overbar{\Delta}^+ \Z_0^n$, leading to Theorems~\ref{t.gap.intro} and~\ref{t.inverse.intro}.
Section~\ref{The Smoothing Parameter} proves Theorem~\ref{t.smoothing.intro} and thereby obtains sharp bounds on mixing time in terms of the number of vertices and maximum degree of $G$.

The final Section \ref{Examples} treats several examples: cycles, complete and complete bipartite graphs, the torus, and `rooted sums.' For the complete graph, we show a lower bound on mixing time that matches the upper bound from Section \ref{The Smoothing Parameter}, proving Theorem \ref{t.cutoff.intro}. The examples are collected in one place to improve the flow of the paper, but some readers will want to look at them before or in parallel to reading the proofs in Sections \ref{Dual Lattices} and \ref{The Smoothing Parameter}.

\subsection{Notation}

Throughout the paper, $\N = \{0,1,2,\ldots\}$ and $\log$ is the natural logarithm. We denote the usual inner product and Euclidean norm on $\R^n$ by $\langle \cdot, \cdot \rangle$ and $\| \cdot \|_2$. We use standard Landau notation: $f(n)=o(g(n))$ if $\lim_{n\to\infty}f(n)/g(n)=0$; $f(n)=O(g(n))$ if $f(n)\leq Cg(n)$ for some $C < \infty$ and all $n$; $f(n)=\Omega(g(n))$ if $f(n)\geq cg(n)$ for some $c > 0$ and all $n$; and $f(n)=\Theta(g(n))$ if $f(n)=O(g(n))$ and $f(n)=\Omega(g(n))$. At times we write $f(n) \ll g(n)$ to mean $f(n) = o(g(n))$.

\section{The Sandpile Chain}
\label{The Sandpile Chain}

We begin by formally defining the sandpile chain.
Let $G=(V,E)$ be a simple connected graph with finite vertex set $V=\{v_{1},\ldots,v_{n}\}$. 
We call $v_{n}$  the \emph{sink} and write $s=v_{n}$. 
A sandpile is a collection of indistinguishable \emph{chips} distributed amongst the non-sink vertices $\V=V\setminus\{s\}$, 
and thus can be represented by a function $\eta:\V\to\N$. The \emph{configuration} $\eta$ corresponds to the sandpile 
with $\eta(v)$ chips at vertex $v$. We say that $\eta$ is \emph{stable at the vertex $v\in\V$} if $\eta(v)<\deg(v)$, 
and say that $\eta$ is \emph{stable} if it is stable at each non-sink vertex. If the number of chips at $v$ is greater than or equal to its degree, then the vertex is allowed to \emph{topple}, sending one chip to each of its neighbors. This leads to the new configuration 
$\eta'$ with $\eta'(v)=\eta(v)-\deg(v)$ and $\eta'(u)=\eta(u)+1$ if $\{u,v\}\in E$, and $\eta'(u)=\eta(u)$ otherwise. 
Toppling $v$ may cause other vertices to become unstable, which can lead to further topplings, and any chip that falls into the sink is gone forever. 
Since we are assuming that $G$ is connected, the presence of the sink ensures that one can reach a stable sandpile from any initial configuration 
(in finitely many steps) by successively performing topplings at unstable sites. An easy argument shows that the final stable configuration, which we denote by 
$\eta^{\circ}$, does not depend on the order in which the topplings are carried out, hence the appelation abelian sandpile \cite{Dha}.

Define the sum of two configurations by $(\sigma+\eta)(v)=\sigma(v)+\eta(v)$. The abelian property shows that if we restrict our attention to the set of stable configurations 
$\mathcal{S}=\{\eta\in\N^{\V}:\eta(v)<\deg(v)\textrm{ for all }v\in\V\}$,
then the operation of addition followed by stabilization, $\eta\oplus\sigma=(\eta+\sigma)^{\circ}$, makes $\mathcal{S}$ into a commutative monoid. 
The identity is the empty configuration $\iota\equiv 0$. In light of this semigroup structure, it is natural to consider random walks on $\mathcal{S}$: 
If $\mu$ is a probability on $\mathcal{S}$, then beginning with some initial state $\eta_{0}$, define $\eta_{t+1}=\eta_{t}\oplus\sigma_{t+1}$ where $\sigma_{1},\sigma_{2},\ldots$ are drawn independently from $\mu$. A natural candidate for $\mu$ is the uniform distribution on the configurations $\delta_{v}(u)=1\{u=v\}$ as $v$ ranges over $V$. (Note that $\delta_{s}=\iota$.) In words, at each time step we add a chip to a random vertex and stabilize.

Now define the \emph{saturated configuration} by $\eta_{\ast}(v)=\deg(v)-1$ for each $v\in\V$. Our assumptions ensure that from any initial state, the random walk on $\mathcal{S}$ driven by the uniform distribution on $\{\delta_{v}\}_{v\in V}$ will visit $\eta_{\ast}$ in finitely many steps with full probability. (If $M=\sum_{v\in\V}\left(\deg(v)-1\right)$, then the probability of visiting $\eta_{\ast}$ within $M$ steps from any stable configuration is at least $n^{-M}$.) The random walk will thus eventually be absorbed by the communicating class 
$\G=\{\eta\in\mathcal{S}:\eta=\eta_{\ast}\oplus\sigma\textrm{ for some }\sigma\in\mathcal{S}\}$. 
Accordingly, the configurations in $\G$ are called \emph{recurrent}.
In the language of semigroups, $\G=\eta_{\ast}\oplus\mathcal{S}$ is the minimal ideal of $\mathcal{S}$. To see that this is so, suppose that $\mathcal{I}$ is an ideal of $\mathcal{S}$ (that is, $\mathcal{I} \oplus \mathcal{S} \subseteq \mathcal{I}$) and let $\sigma\in\mathcal{I}$. Define $\overline{\sigma}(v)=\deg(v)-1-\sigma(v)$. 
Then $\overline{\sigma}\in\mathcal{S}$, so 
$\eta_{\ast}=\sigma+\overline{\sigma}=\sigma\oplus\overline{\sigma}\in\mathcal{I}\oplus\mathcal{S}\subseteq\mathcal{I}$,
hence $\G\subseteq\mathcal{I}$. As the minimal ideal of a commutative semigroup, $\G$ is a nonempty abelian group under $\oplus$. 
(Briefly, for any $a\in\G$, $\G\subseteq a\oplus\G\subseteq \G$ since $\G$ is the minimal ideal, hence $a\oplus x=b$ has a solution in $\G$ for all $a,b\in\G$.) The article \cite{BabTou} contains an excellent exposition of this perspective. 

Because the random walk will eventually end up in the \emph{sandpile group} $\G$ anyway, it makes sense to restrict the state space to $\G$ to begin with. Rather than thinking of this Markov chain in terms of $\mathcal{S}$ acting on $\G$, it is more convenient to consider it as
a random walk on $\G$ so that we may draw on a rich existing theory. 
To this end, let $id$ denote the identity in $\G$ (which is not equal to $\iota$ in general). 
Then for each $v\in V$, $\sigma_{v}:=\delta_{v}\oplus id\in\G$ since $\G$
is an ideal of $\mathcal{S}$ containing $id$. 
Also, $\sigma_{v}\oplus\eta=\delta_{v}\oplus id\oplus\eta=\delta_{v}\oplus\eta$ for all $\eta\in\G$. 
The process of successively adding chips to random vertices and stabilizing can thus be represented as the random walk on $\G$ driven by the uniform distribution on 
$S=\{\sigma_{v}\}_{v\in V}$.\footnote{More precisely, the pushforward of the uniform distribution on $V$ under the map $v \mapsto \sigma_v$. Lemma 4.6 in \cite{BN07} implies that the $\sigma_{v}$'s are distinct as elements of $\G$ if and only if $G$ is 2-edge connected.}
Since $\G$ is an abelian group generated by $S$, we can conclude, for example, that the chain is irreducible with uniform stationary distribution and that the characters of $\G$ form an orthonormal basis of eigenfunctions for the transition operator \cite{SalCo}. 

Though the foregoing is all very nice in theory, even computing the identity element of the sandpile group of a specific graph in these abstract terms is typically quite involved, so it is useful to establish a more concrete realization of $\G$. Recall that the reduced Laplacian $\Delta$ of $G$ is the $(n-1) \times (n-1)$ submatrix of the full Laplacian $\overbar{\Delta}$ formed by deleting the $n$th row and column (corresponding to the sink vertex). We claim that 
	\[ \G\cong\Z^{n-1}/\Delta\Z^{n-1}, \] 
an isomorphism being given by $\eta\mapsto (\eta(v_{1}),\ldots,\eta(v_{n-1}))+\Delta\Z^{n-1}$.
Note that this implies that $\left|\G\right|=\textrm{det}(\Delta)$, which is equal to the number of spanning trees in $G$ by the matrix-tree theorem \cite{Stan}.
The interpretation is that $z\in\Z^{n-1}$ corresponds to the configuration having $z_{i}$ chips at vertex $v_{i}$, where we are allowing vertices to have a negative number of chips---a hole or a debt, say. For $x\in\Z^{n-1}$, adding $\Delta x$ to $z$ corresponds to performing $-x_{i}$ topplings at each vertex $v_{i}$. (A negative toppling means that the vertex takes one chip from each of its neighbors.) 
Isomorphism is established by showing that each coset contains exactly one vector corresponding to a recurrent configuration. One way to see this is to observe 
that the lattice $\Delta\Z^{n-1}$ contains points with arbitrarily large smallest coordinate, so from any $z\in\Z^{n-1}$, there is an 
$x\in\Delta\Z^{n-1}$ with $(x+z)_{i}\geq\deg(v_{i})$ for all $1\leq i\leq n-1$. Then use the fact that stabilizing such a configuration amounts to adding some 
$y\in\Delta\Z^{n-1}$ and results in a unique recurrent configuration. We refer the reader to \cite{HLMPPW} for a more detailed account.

In this view, the sandpile chain is a random walk on $\Gamma:=\Z^{n-1}/\Delta\Z^{n-1}$ driven by the uniform distribution on $\{e_1, \ldots, e_{n-1},\zero\}$, where $e_{i}$ is the vector with a one in the $i$th coordinate and zeros elsewhere. Geometrically, we are performing a random walk on the positive orthant of $\Z^{n-1}$ by taking steps of unit length in a direction chosen uniformly at random from the standard basis vectors of $\Z^{n-1}$ (with a holding probability of $1/n$), but we are concerned only with our relative location within cells of the lattice $\Delta\Z^{n-1}$.

\subsection{Spectral properties}
\label{Spectral properties}
 
Since the sandpile chain is a random walk on a finite abelian group $\Gamma$, we can find the eigenvalues and eigenfunctions of the associated transition matrix in terms of the \emph{characters} of $\Gamma$, that is, elements of the dual group
	\[ \widehat{\Gamma} := \Hom(\Gamma,\T) \]
where $\T$ is the set of complex numbers of modulus $1$. We emphasize that the sandpile chain is \emph{not} reversible in general, as our generating set $\{e_1, \ldots, e_{n-1},\zero\}$ is generally not closed under negation modulo $\Delta\Z^{n-1}$. Nevertheless we will see that the usual bounds on mixing still hold due to orthogonality of the characters.

Our starting point is the following well-known lemma, which is particularly simple in our case thanks to the fact that all irreducible representations of an abelian group are one-dimensional. See \cite{Dia} for the general (nonabelian) case. 

\begin{lemma}
\label{ab_gp_eigvals}
If $\mu$ is a probability on a finite abelian group $A$, then the transition matrix for the associated random walk has an orthonormal basis of eigenfunctions consisting of the characters of $A$. The eigenvalue corresponding to a character $\chi$ is given by the evaluation of the Fourier transform 
\[
\widehat{\mu}(\chi) := \sum_{a\in A}\mu(a)\chi(a).
\]
\end{lemma}

\begin{proof}
Let $Q(x,y)=\mu(yx^{-1})$ denote the transition matrix. For any character $\chi$, we have 
\begin{align*} 
(Q\chi)(x) & =\sum_{y\in A}Q(x,y)\chi(y)=\sum_{y\in A}\mu(yx^{-1})\chi(y)\\
 & =\sum_{z\in A}\mu(z)\chi(zx)=\chi(x)\sum_{z\in A}\mu(z)\chi(z),
\end{align*}
hence $\chi$ is an eigenfunction with eigenvalue $\widehat{\mu}(\chi)$. The result follows by observing that there are $|A|$ characters and they are orthonormal with respect to the standard inner product 
$\left(f,g\right)=\frac{1}{|A|}\sum_{a\in A}f(a)\overline{g(a)}$ on $\C^{A}$. 
\end{proof}

To apply the preceding lemma, we are going to express the characters of the abstract group $\Gamma$ in terms of functions defined on the vertices of our graph $G$.
These are the \emph{multiplicative harmonic functions}
	\[ h : V \to \T \]
satisfying $h(s)=1$ and
	\begin{equation} \label{e.geomean} h(v)^{\deg(v)} = \prod_{u\sim v}h(u) \end{equation}
for all $v \in V$.  We will refer to \eqref{e.geomean} as the `geometric mean value property' at vertex $v$. We pause to make two remarks about this definition.

\begin{remark}
\label{codomain}
We could equivalently take the codomain of $h$ to be all nonzero complex numbers. The geometric mean value property implies $\overbar \Delta (\log |h|) \equiv 0$, so $\log |h|$ is a constant function. Since $h(s)=1$ it follows that $h$ takes values on the unit circle $\T$.
\end{remark}

\begin{remark}
\label{GMVP-all-but-one}
The identity 
	\begin{equation} \label{e.eulerian} \prod_{v \in V} h(v)^{\deg(v)}  = \prod_{v \in V} \prod_{u \sim v} h(u) \end{equation}
holds for \emph{any} function $h$. Therefore if the geometric mean value property holds at every vertex but one, then in fact it holds at every vertex.
\end{remark}

Denote by $\HH$ the set of all multiplicative harmonic functions on $G$. This set is nonempty since it contains the constant function $1$. One readily checks that it is an abelian group under pointwise multiplication.

For each $h\in\HH$, define $\chi_{h}':\Z^{n-1}\to\T$ by 
\begin{equation}
\label{chi-h-prime}
\chi_{h}'(z)=\prod_{j=1}^{n-1}h(v_{j})^{z_{j}},
\end{equation}
and define $\chi_{h}:\Z^{n-1}/\Delta\Z^{n-1}\to\T$ by 
$\chi_{h}(z+\Delta\Z^{n-1})=\chi_{h}'(z)$. 
The ensuing proof shows that $\chi_{h}$ is well-defined.

\begin{lemma}
\label{irr_chars}
The characters of $\Gamma$ are precisely $\{\chi_{h}\}_{h\in\HH}$.
\end{lemma}

\begin{proof}
From \eqref{chi-h-prime}, each function $\chi'_h$ is a homomorphism. For each standard basis vector $e_j \in \Z^{n-1}$,
\begin{equation}
\label{chi-geom-meanvalue}
\chi'_h(\Delta e_j) = h(v_j)^{\deg(v_j)} \prod_{v_k \sim v_j} h(v_k)^{-1} = 1,
\end{equation}
using the geometric mean value property and $h(s) = 1$. Therefore $\chi'_h$ is constant on cosets of 
$\Delta \Z^{n-1}$ and thus descends to $\chi_h \in \widehat{\Gamma}$.

Conversely, every $\chi \in \widehat{\Gamma}$ lifts to a homomorphism $\chi': \Z^{n-1} \to \T$ that is identically $1$ on $\Delta \Z^{n-1}$. Define $h: V \to \T$ by $h(v_j) = \chi'(e_j)$ for $j \leq n-1$ and $h(s) = 1$. Since $\chi'$ is a homomorphism, it satisfies equation \eqref{chi-h-prime}, so $\chi = \chi_h$. By equation \eqref{chi-geom-meanvalue}, $h$ satisfies the geometric mean value property at every non-sink vertex, implying that $h \in \HH$.
\end{proof}

Observe that the map $h \mapsto \chi_h$ gives an isomorphism between $\HH$ and $\widehat{\Gamma}$. Since $\Gamma$ is a finite abelian group, we have $\HH\cong\widehat{\Gamma}\cong\Gamma\cong\G$. 
Thus examining the multiplicative harmonic functions on $G$ can give insights about the structure of the corresponding sandpile group. For example, 
the following proposition gives another way to see that the sandpile group 
of an undirected graph 
does not depend on the choice of sink.

\begin{prop}
\label{sink_change}
If $\HH_{v}$ denotes the group of multiplicative harmonic functions on $G$ with sink at $v$, then the map 
$\phi:\HH_{u}\to\HH_{w}$ given by $\phi(h)(v)=h(w)^{-1}h(v)$ is an isomorphism.
\end{prop}

Returning to the language of sandpiles, we see that the characters of $\G$ are the functions $\{f_h : \G \to \T\}_{h \in \HH}$, where
	\[ f_{h}(\eta)=\prod_{v\in\V}h(v)^{\eta(v)}. \]
Thus Lemma \ref{ab_gp_eigvals} implies the following.

\begin{theorem}
\label{sandpile_eigvals} 
An orthonormal basis of eigenfunctions for the transition matrix of the sandpile chain on $G$ is $\{f_{h}\}_{h\in\HH}$. 
The eigenvalue associated with $f_{h}$ is 
	\[ \lambda_{h} := \frac{1}{n}\sum_{v\in V}h(v). \]
\end{theorem}

\begin{proof}
Letting $p$ denote the uniform distribution on $\{\sigma_{v}\}_{v\in V}$, we see that the eigenvalue associated with $f_{h}$ is 
\[
\widehat{p}(f_{h})=\sum_{\eta\in\G}p(\eta)f_{h}(\eta)
=\sum_{v\in V}\frac{1}{n}f_{h}(\sigma_{v})=\frac{1}{n}\sum_{v\in V}h(v). \qedhere
\]
\end{proof}

Before proceeding to our main topic of mixing times, we remark on two ways to extend the above analysis.

\begin{remark}
\label{arbitrary-chip-addition}
Theorem~\ref{sandpile_eigvals} holds for an arbitrary chip-addition distribution $\mu$, with the only change being that $f_{h}$ then has eigenvalue $\lambda_{h}=\sum_{v\in V}\mu(v)h(v)$. For example, we may choose to add chips at a uniform \emph{nonsink} vertex, in which case the eigenvalue associated with $f_{h}$ is 
	\[ \widetilde{\lambda}_{h} = \frac{1}{n-1}\sum_{v\in \V}h(v). \] 
Since the only multiplicative harmonic function that is constant on $\V$ is $h \equiv 1$ (as can be seen by applying the geometric mean value property at a neighbor of the sink), we have $\big| \widetilde{\lambda}_h \big| < 1$ for all nontrivial $h \in \HH$. Therefore, this `non-lazy' version of the sandpile chain is aperiodic despite the lack of holding probabilities.
\end{remark}

\begin{remark}
\label{directed-graph}
Suppose that $G$ is a directed multigraph in which every vertex has a directed path to the sink. In this setting there is a Laplacian $\overbar{\Delta}$ analogous to \eqref{e.thelaplacian} although it is no longer a symmetric matrix: its off-diagonal entries are $\overbar{\Delta}_{uv} = -e_{uv}$ where $e_{uv}$ is the number of edges directed from $u$ to $v$, and its diagonal entries are $\overbar{\Delta}_{uu} = d_u-e_{uu}$ where $d_u$ is the outdegree of vertex $u$. The sandpile group is 
naturally isomorphic to $\Z^{n-1} / \Delta_s^\transpose \Z^{n-1}$, where $\Delta_s^\transpose$ is the transpose of the submatrix omitting the row and column of $\overbar{\Delta}$ corresponding to the sink $s$. Unlike in the undirected case, this group may depend on the choice of $s$ (but see \cite[Theorem 2.10]{FL15}). If we define
\[
\HH_s := \Big\{h:V\to \T\textrm{ such that }h(s)=1\textrm{ and }h(u)^{d_u}
=\prod_{v\in V}h(v)^{e_{uv}}\textrm{ for all }u\neq s \Big\}
\] 
then nearly all of this section carries over to the directed case. The exceptions are \eqref{e.eulerian} and Proposition~\ref{sink_change}, which hold only when $G$ is Eulerian.
\end{remark}

\subsection{Mixing times}
\label{Mixing times}

Since the sandpile chain is an irreducible and aperiodic random walk on a finite group, the law of the chain at time $t$ approaches the uniform distribution on $\G$ as $t\to\infty$. Our interest is in the rate of convergence. To avoid trivialities, we assume throughout that $G$ is not a tree, so that $|\G| > 1$.

The metrics we consider are the \emph{$L^{2}$ distance} 
\[
\|\mu-\nu\|_{2}=\Big(|\G|\sum_{g\in\G}|\mu(g)-\nu(g)|^{2}\Big)^{\frac{1}{2}}
\]
and the \emph{total variation distance}
\[
\|\mu-\nu\|_{\TV}=\frac{1}{2}\sum_{g\in\G}|\mu(g)-\nu(g)|
=\max_{A\subseteq\G}\left(\mu(A)-\nu(A)\right).
\]
Note that Cauchy-Schwarz immediately implies $\|\mu-\nu\|_{\TV}\leq \frac{1}{2}\|\mu-\nu\|_{2}$.

We always assume that the chain is started from a deterministic state $\sigma \in \G$. 
As a random walk on a group, the distance to stationarity after $t$ steps under either of these metrics 
is independent of $\sigma$, so without loss of generality we take $\sigma=id$ henceforth.  Writing $P_{id}^{t}$ for the distribution of the sandpile chain at time $t$ and $U$ for the uniform distribution on $\G$, the following lemma (proved in \cite{SalCo}) shows that 
$\|P_{id}^{t}-U\|_{2}$ is completely determined by the eigenvalues from Theorem \ref{sandpile_eigvals}.

\begin{lemma}
\label{Fourier}
Let $Q_{id}^t$ be the $t$-step distribution of a random walk on a finite abelian group $A$ started at 
the identity and driven by a probability measure $\mu$, and let $\pi$ be the uniform distribution on $A$.
Then
\[
\|Q_{id}^{t}-\pi\|_{2}^{2}=\sum_{\chi\neq 1}|\widehat{\mu}(\chi)|^{2t}
\]
where the sum is over all nontrivial characters. 
\end{lemma}

\noindent For the sandpile chain, this says that 
\begin{equation}
\label{L2_equality}
\|P_{id}^{t}-U\|_{2}^{2}=\sum_{h\in\HH\setminus\{1\}}|\lambda_{h}|^{2t}.
\end{equation}

Lemma \ref{Fourier} is a special case of the famous Fourier bound from \cite{Dia}. It also follows 
by an eigenfunction expansion exactly as in the standard proof of the spectral bound for reversible chains 
(see \cite[Ch.\ 12]{LPW}). Though sandpile chains are not reversible in general, many of the same arguments still apply because of the orthogonality of the eigenfunctions. Probabilistically, this orthogonality is equivalent to the statement that the transition operator commutes with its time reversal.

Total variation convergence rates can also be estimated in terms of spectral information. 
To make this precise, we introduce some terminology. Let
$\lambda_{\ast}=\max_{h\in\HH\setminus\{1\}}|\lambda_{h}|$ denote the size of the subdominant 
eigenvalue of $P$. The (discrete time) \emph{spectral gap} is defined as $\gamma_d=1-\lambda_{\ast}$, and the \emph{relaxation time} as $\trel=\gamma_d^{-1}$.
The \emph{mixing time} is  
$\tmix(\e)=\min\{t\in\N : \|P_{id}^{t}-U\|_{\TV}\leq\e\}$.
The following proposition is standard for reversible Markov chains. It holds in our context as well 
using the above definition of relaxation time.

\begin{prop}
\label{TV_eigval}
The mixing time of the sandpile chain satisfies
\[
\log\left(\frac{1}{2\e}\right)(\trel-1)\leq \tmix(\e)
\leq \left\lceil\log\left(\frac{|\G|^{\frac{1}{2}}}{2\e}\right)\trel\right\rceil.
\]
\end{prop}

\begin{proof}
The upper bound follows from \eqref{L2_equality} and the $L^2$ bound on total variation by bounding the summands with $\lambda_{\ast}^{2t}$. Arithmetic manipulations then give the mixing time bound; see \cite[Ch.\ 12]{LPW} for details.

The lower bound is Theorem 12.4 from \cite{LPW}. Reversibility is used there only to conclude that the constant function is orthogonal to nontrivial eigenfunctions in the 
inner product weighted by the stationary distribution. However, the same statement holds also for nonreversible chains since the stationary distribution, as a left eigenfunction with eigenvalue $1$, is orthogonal to all nontrivial right eigenfunctions under the standard inner product.
\end{proof}

Our next result gives a universal lower bound on the spectral gap $\gamma_d$ (and thus an upper bound on $\trel$) in terms of the number of vertices $n$ and the second largest degree $d_*$ of the underlying graph.

\begin{theorem}
\label{lazy_gap_bd}
Suppose that $G$ has degree sequence $d_{1}\leq\cdots\leq d_{n-1}\leq d_{n}$ and set $d_{\ast}=d_{n-1}$. Then the spectral gap of the (discrete time) sandpile chain on $G$ satisfies 
\[
\gamma_d\geq\frac{8}{d_{\ast}^{2}n}.
\]
\end{theorem}

The proof uses the following inequality.

\begin{lemma} 
\label{cos-upper-bound}
Suppose $0 \leq r \leq 2\pi$. Then $\cos(x) \leq 1 - cx^2$ for all $|x| \leq r$, where $c =\frac{1-\cos(r)}{r^{2}}$.
\end{lemma}

\begin{proof}
Let $f(x) = [1 - \cos(x)]\big/x^2$, with $f(0) = 1/2$ so that $f$ is continuous on~$\R$. We will show that $f$ is decreasing on the interval $[0,2\pi]$. The inequality $f(x) \geq f(r)$ for $0 \leq x \leq r \leq 2\pi$ rearranges into the desired inequality $\cos(x) \leq 1 - cx^2$. Since $f$ is even, the inequality also holds when $-2\pi \leq -r \leq x \leq 0$.

We have $f'(x) = g(x)/x^3$, where $g(x) = x\sin(x) - 2(1 - \cos(x))$. Thus, it suffices to show that $g$ is negative on the interval $(0, 2\pi)$. Note that $g(0) = g(2\pi) = 0$. As well, $g'(x) = x\cos(x) - \sin(x)$ and $g''(x) = -x\sin(x)$. This means $g'$ is strictly decreasing for $x \in [0,\pi]$ and strictly increasing for $x \in [\pi,2\pi]$. Since $g'(0) = 0$ and $g'(2\pi) = 2\pi$, there is $x_0 \in (\pi,2\pi)$ such that $g'(x_0) = 0$ and $g'(x) < 0$ for $x \in (0,x_0)$ and $g'(x) > 0$ for $x \in (x_0,2\pi]$. Therefore $g$ is strictly decreasing for $x \in [0,x_0]$ and strictly increasing for $x \in [x_0,2\pi]$. Because $g(0) = g(2\pi) = 0$, we conclude that $g(x) < 0$ for $x \in (0,2\pi)$, as desired.
\end{proof}

\begin{proof}[Proof of Theorem~\ref{lazy_gap_bd}]
We note at the outset that Proposition \ref{sink_change} implies that the moduli of the eigenvalues 
are invariant under change of sink, so we may assume without loss of generality that $s$ is located at 
a vertex of maximum degree. Accordingly, $\deg(v)\leq d_{\ast}$ for all $v\in\V$.

Now fix $h\in\HH$ and suppose that there exists an arc 
$C_{ab}=\{e^{i\theta}:2\pi a\leq\theta\leq 2\pi b\}$ with $0<b-a<1/d_{\ast}$ such that 
$h(v)\in C_{ab}$ for every $v\in V$. We will show this implies $h \equiv 1$. Write $h(v)=e^{2\pi i g(v)}$, where 
$g:V\to [a,b]$.  Note that for all $v \in V$,
	\[ \overbar{\Delta} g (v) = \sum_{w \sim v} (g(v)-g(w)) \]
is an integer, by the geometric mean value property of $h$ at $v$. On the other hand, since $g(V)\subseteq [a,b]$,  
\[
\left| \overbar{\Delta} g (v)\right|\leq\sum_{w \sim v}\left|g(v) - g(w)\right| \leq d_{\ast} (b-a) < 1
\]
for all $v\in\V$. Since the left side is an integer it must be zero, so $g$ is harmonic on $\V$ in the usual sense. 
Uniqueness of harmonic extensions implies $g\equiv g(s)$ and thus $h\equiv h(s)=1$, as desired.

For any fixed $h \in\HH\setminus\{1\}$ write $\lambda_h = re^{i \theta}$ with $r \geq 0$. 
The preceding argument shows that $h(V)$ is not contained in any segment of the unit circle having 
arc length less than $2\pi/d_{\ast}$.
For each $v \in V$, let $A(v)$ be the unique angle $-\pi < \phi \leq \pi$ such that 
$h(v) = e^{i(\theta + \phi)}$. Let $\phi_1 = \min_{v \in V} A(v)$ and $\phi_2 = \max_{v \in V} A(v)$, 
so that $\phi_2 - \phi_1 \geq 2\pi / d_\ast$. We have
\[
|\lambda_h| = \frac{1}{n} \sum_{v \in V} \cos(A(v)) \leq \frac{1}{n} \left[ n-2 + \cos(\phi_1) + \cos(\phi_2) \right].
\]
Using the identity
\[
\cos(\phi_1) + \cos(\phi_2) = 2 \cos\left( \frac{\phi_2 + \phi_1}{2} \right) \cos\left( \frac{\phi_2 - \phi_1}{2} \right) \leq 2 \cos\left( \frac{\phi_2 - \phi_1}{2} \right),
\]
along with $\pi / d_\ast \leq (\phi_2 - \phi_1)/2 < \pi$ and the fact that  cosine  is 
decreasing on $[0,\pi]$, we have
\[
|\lambda_h| \leq 
\frac{1}{n} \left[ n-2 + 2 \cos\left( \frac{\pi}{d_\ast} \right) \right].
\]
Recalling our assumption that $G$ is not a tree, $\pi/d_{\ast}\leq\pi/2$, so Lemma \ref{cos-upper-bound} gives
\[
\left|\lambda_{h}\right|   \leq\frac{1}{n}\left[n-2+2\left(1-\frac{4}{\pi^{2}}\left(\frac{\pi}{d_{\ast}}\right)^{2}\right)\right]=1-\frac{8}{d_{\ast}^{2}n}. \qedhere
\]
\end{proof}

Note that the number of stable configurations, and thus the order of the sandpile group, is at most $d_{\ast}^{n-1}$, so it follows from Theorem \ref{lazy_gap_bd} and Proposition \ref{TV_eigval} that 
$\tmix(\e)=O\left((d_{\ast}^{2}\log d_\ast)n^{2}\right)$. 
In Section \ref{The Smoothing Parameter}, we will show that this bound can be improved to  $O\left(d_{\ast}^{2}n\log n \right)$.

\subsection{Gadgets}
\label{Gadgets}

In many cases, one can determine the order of the relaxation time $\trel$ of the sandpile chain by constructing just a single multiplicative harmonic function. For bounded degree graphs, Theorem \ref{lazy_gap_bd} shows that $\trel=O(n)$. Thus to show that $\trel=\Theta(n)$, it suffices to find an $h\in\HH$ with $|\lambda_{h}|\geq 1-C/n$ for some constant $C$. Since the eigenvalues are all of the form $\lambda_{h}=\frac{1}{n}\sum_{v\in V}h(v)$, we see that large eigenvalues correspond to `nearly constant' multiplicative harmonic functions. In particular, if there is an $h\in\HH$ which is constant on $U\subseteq V$ with $|V\setminus U|=m$, then 
$|\lambda_{h}|\geq (n-2m)/n$. 

For example, consider the $m$-fold Sierpinski gasket graph $SG_{2}(m)$ defined recursively as follows: $SG_{2}(0)$ is the triangle $K_{3}$ and $SG_{2}(m)$ for $m \geq 1$ is obtained by gluing three copies of $SG_{2}(m-1)$ to obtain a triangle with center cut out. We take the `topmost' vertex as the sink (Figure \ref{sierpinski}).

\begin{figure}[h]
	\centering
	\includegraphics[scale=0.9]{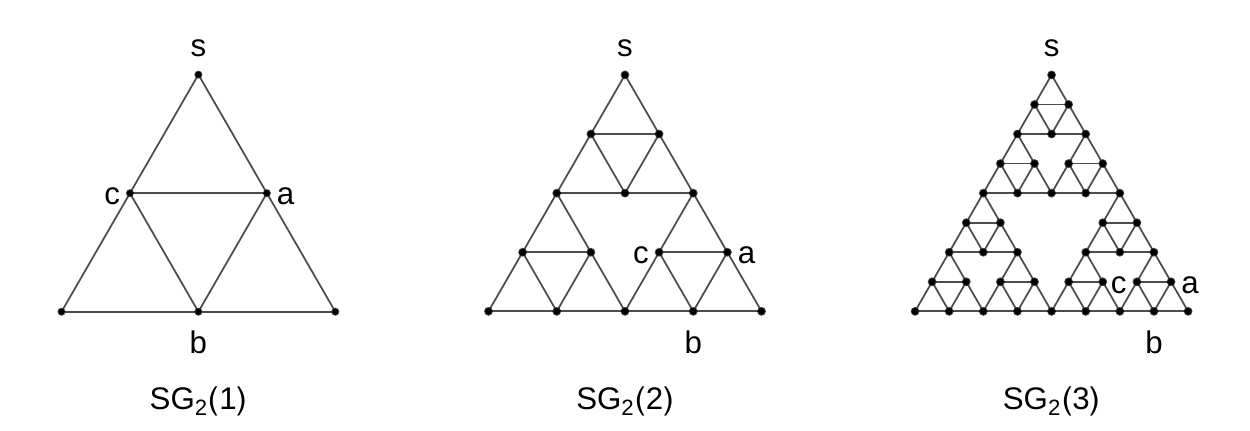}
	\caption{Sierpi\'{n}ski Gasket Graphs}
	\label{sierpinski}
\end{figure}

It is easy to see that $SG_{2}(m)$ has $\left|V(m)\right|=\frac{1}{2}\left(3^{m+1}+3\right)$ vertices, and it is known \cite{CCY} that the number of spanning trees is $\left|\G\left(SG_{2}(m)\right)\right|=2^{\alpha(m)} 3^{\beta(m)} 5^{\gamma(m)}$ with $\alpha(m)=\frac{1}{2}\left(3^{m}-1\right)$, $\beta(m)=\frac{1}{4}\left(3^{m+1}+2m+1\right)$, $\gamma(m)=\frac{1}{4}\left(3^{m}-2m-1\right)$. 
To the best of the authors' knowledge, the invariant factor decomposition of $\G\left(SG_{2}(m)\right)$ is an open problem.

Each $SG_{2}(m)$ contains a copy of $SG_{2}(1)$ in the lower right corner. Referring to Figure \ref{sierpinski}, define $h$ by $h(a)=h(b)=h(c)=-1$ where $a,b,c$ are the three inner vertices of this copy of $SG_2(1)$; and $h(v)=1$ for all other vertices $v$ of $SG_2(m)$. Then $h$ is multiplicative harmonic with associated eigenvalue $\lambda_{h}=1-6/|V(m)|$, hence the relaxation time is $\Omega\left(\left|V(m)\right|\right)$. Since $SG_2(m)$ has bounded degree, we conclude that $\trel = \Theta(|V(m)|)$.

Because large eigenvalues so often arise in this fashion, it is useful to introduce the following definition.

\begin{definition}
Let $G' = (V',E')$ be a vertex induced subgraph of $G$, and let $h'$ be a function from $V'$ to $\T$. We call $(G',h')$ a \emph{gadget} of size $m > 0$ in $G$ if:

\begin{enumerate}
\item $h'$ satisfies the geometric mean value property (with respect to $G'$) at every $v \in V'$.

\item The \emph{interior} $\interior(V') = \{v \in V' \st h'(v) \neq 1\}$ has size $m$.

\item The \emph{boundary} $\partial V' = V' \setminus \interior(V')$ contains the subgraph boundary \\
$\{v' \in V' \st \{v',u\} \in E \text{ for some } u \in V \setminus V' \}$.
\end{enumerate}
\end{definition}

Often we refer to the gadget $(G', h')$ simply as $G'$ for convenience. For example, the Sierpi\'{n}ski gasket graph $SG_2(m)$ has a gadget $SG_{2}(1)$ of size $\left|\{a,b,c\}\right|=3$.

\begin{prop}
\label{gadget-prop}
If $(G',h')$ is a gadget of size $m$ in $G$, then $\gamma_d \leq \frac{2m}{n}$.
\end{prop}

\begin{proof}
Since $\gamma_d$ is independent of the choice of sink vertex, we may assume that $s \notin \interior(V')$. Define $h: V \to \T$ by $h \equiv h'$ on $V'$ and $h \equiv 1$ on $V \setminus V'$. It is easily checked that $h \in \HH(G)$ and $|\lambda_h| \geq \Re(\lambda_h) \geq 1 - \frac{2m}{n}$.
\end{proof}

The next example shows that a gadget can have size as small as $2$. Suppose that $v_{1},v_{2}\in \V$ have common neighborhood $N=\{u\in V:u\sim v_{1}\}=\{u\in V:u\sim v_{2}\}$ with 
$d=\left|N\right|>1$. Then the induced subgraph with vertex set $V' = \interior(V') \cup \partial V'$, $\interior(V') = \{v_1,v_2\}$, $\partial V' = N$, is a gadget of size $2$ in $G$. 
Indeed, if $\omega$ is any nontrivial $d$th root of unity, then the function $h:V\to \T$ given by 
\[ 
h(v)=\left\{ \begin{array}{cc}
\omega, & v=v_{1}\\
\omega^{-1}, & v=v_{2}\\
1, & \textrm{else}
\end{array}\right.
\]
is multiplicative harmonic on $G$. Taking $\omega = e^{2\pi i/d}$, the eigenvalue corresponding to $h$ is $\lambda_{h}=1-\frac{2}{n}\left(1-\cos\left(\frac{2\pi}{d}\right)\right)$. Thus, in any graph on $n$ vertices, two of which have the same neighborhood of size $d$, the spectral gap of the discrete time sandpile chain has order at most $1 / (d^2 n)$. 

The common neighborhood gadget can also be understood from a more probabilistic perspective. The eigenfunction corresponding to $h$ gives information about the difference mod $d$ between the number of chips at $v_{1}$ and $v_{2}$. For the chain to equilibrate, it must run long enough for this mod $d$ difference to randomize. Since $v_{1}$ and $v_{2}$ have all neighbors in common, the mod $d$ difference is invariant under toppling; it changes only when a chip is added at $v_1$ or $v_2$.

Small gadgets can drastically affect the mixing time of the sandpile chain. For example, the sandpile chain on the cycle $C_{n}$ mixes completely after a single step (see Section \ref{Cycle graph}); but if we add just two extra vertices $u$, $w$ and $2d$ extra edges $\{u,v_{j}\}$ and $\{w,v_{j}\}$, $j=1,\ldots,d$ for some $d \geq 2$, 
then the relaxation time of the sandpile chain becomes 
$\Omega(d^{2}n)$.

The following theorem provides one way to rule out the presence of a small gadget.

\begin{theorem}
\label{gadget_girth}
If $G$ has girth $g$, then all gadgets in $G$ have size at least $g/2$.
\end{theorem}

\begin{proof}
We note at the outset that the relevant definitions preclude the existence of gadgets of size $1$, so the theorem is vacuously true when $g\leq 4$.
When $g \geq 5$, assume for the sake of contradiction that $G=(V,E)$ is a counterexample having vertex set of minimum size. Let $(G',h')$ be a gadget in $G$ whose interior $W$ has size $|W| < g/2$, and let $h$ be the extension of $h'$ to $V$ by setting $h \equiv 1$ on $V \setminus V'$, so that $h$ satisfies the geometric mean value property at every $v \in V$. Any vertex of degree $1$ could be deleted without affecting the girth of $G$ or the value of $h$ at the other vertices, so by minimality, every vertex of $G$ has degree at least $2$.

We now define a new graph $K$ on the vertex set $W$. We draw an edge between the distinct vertices $a,b \in W$ if $\{a,b\} \in E$ or if there exists $c\in V\setminus W$ with $\{a,c\},\{b,c\}\in E$. Since any cycle in $K$ would give rise to a cycle in $G$ of length at most $2\left|W\right|$, the graph $K$ has no cycles. Therefore there is $a \in W$ such that $\deg_K(a) = 0$ or $1$.

Since $\deg_G(a) \geq 2$, the vertex $a$ has at least one $G$-neighbor $v \in V \setminus W$. If $v$ has no other neighbors in $W$, then $1 = h(v)^{\deg(v)} = \prod_{w \sim v} h(w) = h(a)$, a contradiction. Therefore $v$ is adjacent to some $b \in W$, which must be the unique $K$-neighbor of $a$. The edge $\{a,b\} \notin E$ because $G$ has no triangles. Thus $a$ has no $G$-neighbors in $W$, and it must have another $G$-neighbor $v' \in V \setminus W$. By the reasoning used for $v$, $v'$ is also adjacent to $b$. But now the edges $\{a,v\},\{v,b\},\{b,v'\},\{v',a\}$ form an illegal $4$-cycle in $G$. This contradiction proves the theorem.
\end{proof}

\section{Dual Lattices}
\label{Dual Lattices}

In Section \ref{The Sandpile Chain} we saw that the sandpile chain is a random walk on the lattice quotient 
	$ \Gamma = \Z^{n-1} / \Delta \Z^{n-1}$ .  
Its character group $\widehat{\Gamma}$ is naturally isomorphic to 
	\[ \widehat{\Gamma} \cong (\Delta^{-1} \Z^{n-1}) / \Z^{n-1}. \] 
We refer to $\Delta^{-1} \Z^{n-1}$ as the \emph{dual Laplacian lattice} because it equals $(\Delta \Z^{n-1})^\ast = \{x \in \R^{n-1} \st \langle x,y \rangle \in \Z \text{ for all } y \in \Delta \Z^{n-1} \}$. Since the group $\HH$ of multiplicative harmonic functions is isomorphic to $\widehat{\Gamma}$ (see Lemma \ref{irr_chars}), each $h \in \HH$ can be identified with an equivalence class $x_h + \Z^{n-1} \subseteq \Delta^{-1} \Z^{n-1}$ of dual lattice vectors.

Given $h \in \HH$, choose $x_h \in \Delta^{-1} \Z^{n-1}$ of minimal Euclidean length that corresponds to $h$. The first main result in this section, Theorem \ref{dual-length-1}, relates the length $\|x_h\|_2$ to the eigenvalue $\lambda_h = \frac{1}{n} \sum_{v \in V} h(v)$ of the sandpile chain. Specifically, the length $\|x_h\|_2$ determines the gap $1 - \Re(\lambda_h)$ up to a constant factor. This property lets us translate information about the lengths of dual lattice vectors into information about the convergence of the \emph{continuous time} sandpile chain.

What about the discrete time sandpile chain? As it turns out, we get parallel results to the continuous case if we use a slightly different dual lattice, constructed using the \emph{pseudoinverse} of the full Laplacian matrix $\overbar{\Delta}$. The pseudoinverse construction leads to a quick proof of Theorem \ref{t.inverse.intro}, which states that if the spectral gap of $\overbar{\Delta}$ is large, then the spectral gap of the discrete time sandpile chain is small.

This section is organized as follows. Section \ref{Discrete and continuous time} provides preliminary details about the discrete and continuous time sandpile chains.  Sections \ref{Dual lattice: Continuous time} and \ref{Dual lattice: Discrete time} construct the two dual lattices and prove the correspondence between dual lattice vector lengths and sandpile chain eigenvalues. Section \ref{An inverse relationship} proves the inverse relationship (Theorem \ref{t.inverse.intro}) and uses it to determine the order of the sandpile chain spectral gap for families of expander graphs.

\subsection{Discrete and continuous time}
\label{Discrete and continuous time}

In this subsection we compare the discrete time sandpile chain with its continuous time analogue. In discrete time, the mixing properties are essentially independent of the choice of sink vertex, while in continuous time, the choice of sink can affect the mixing by up to a factor of $n$.

Recall that the transition matrix $P$ for the discrete time sandpile chain has eigenvalues $\{\lambda_h : h \in \HH\}$, where $\lambda_h = \frac{1}{n} \sum_{v \in V} h(v)$.
The continuous time chain, which proceeds by dropping chips according to a rate $1$ Poisson process, has kernel 
\[
H^{t}(\eta,\sigma) = \sum_{k=0}^{\infty}e^{-t}\frac{t^k}{k!}P^{k}(\eta,\sigma) = e^{-t(I-P)}(\eta,\sigma).
\]
The eigenvalues of $H^t$ are $\{ e^{-t(1 - \lambda_h)} : h \in \HH \}$.

If the chains are started from the identity configuration, then after time $t$ the $L^2$ distances from the uniform distribution $U$ are
\begin{align*}
\| P_{id}^t - U \|_2^2 &= \sum_{h \in \HH \setminus \{1\} } |\lambda_h|^{2t}, \\
\| H_{id}^t - U \|_2^2 &= \sum_{h \in \HH \setminus \{1\}} \left| e^{-t(1 - \lambda_h)} \right|^2 = \sum_{h \in \HH \setminus \{1\}} e^{-2t(1 - \Re(\lambda_h))}.
\end{align*}

The mixing of the discrete time chain is controlled by the eigenvalues with modulus close to $1$, while the mixing of the continuous time chain is controlled by the eigenvalues with real part close to $1$. This motivates the definitions for the discrete and continuous time spectral gaps given in Section \ref{Introduction}:
\begin{align*}
\gamma_d &= \min\{ 1 - |\lambda_h| : h \in \HH, \, h \not\equiv 1 \}, \\
\gamma_c &= \min\{ 1 - \Re(\lambda_h) : h \in \HH, \, h \not\equiv 1 \}.
\end{align*}

\begin{prop} 
\label{gamma-cd-bound}
Every eigenvalue $\lambda_h$ of the discrete time sandpile chain satisfies
\[
1 - |\lambda_h| \leq 1 - \Re(\lambda_h) \leq n(1 - |\lambda_h|).
\]
Hence $\gamma_d \leq \gamma_c \leq n \gamma_d$.
\end{prop}

\begin{proof}
The lower bound is immediate. For the upper bound, since $h(s) = 1$, we can write
\[
\lambda_h = \frac{1}{n}[1 + (n-1)z],
\]
where $z = \frac{1}{n-1} \sum_{v \in \widetilde{V}} h(v)$ satisfies $|z| \leq 1$. Therefore
\[
1 - \Re(\lambda_h) = \frac{n-1}{n} \left[1 - \Re(z) \right]
\]
and
\[
|\lambda_h|^2 = \frac{1}{n^2} \left[ 1 + 2(n-1) \Re(z) + (n-1)^2 |z|^2 \right].
\]
Using that $|z|^2 \leq 1$,
\[
1 - |\lambda_h|^2 \geq \frac{2(n-1)}{n^2} [1 - \Re(z)] = \frac{2}{n} \left[ 1 - \Re(\lambda_h) \right].
\]
Hence
\[
1 - \Re(\lambda_h) \leq \frac{n}{2} (1 + |\lambda_h|)(1 - |\lambda_h|) \leq n(1 - |\lambda_h|). \qedhere
\]
\end{proof}

We have seen that the magnitudes $|\lambda_{h}|$, and thus the values of $\gamma_d$ and $\|P_{id}^{k}-U\|_{2}$, do not depend on the location of the sink. By contrast, the choice of sink can affect the values of $\gamma_c$ and $\| H_{id}^t - U \|_2$. Proposition \ref{gamma-cd-bound} shows that the value of $\gamma_c$ cannot vary by more than a factor of $n$. In Section \ref{Discrete vs Continuous}, we will see two examples where moving the sink changes $\gamma_c$ by a factor of roughly $n/2$.

\subsection{Dual lattice: Continuous time}
\label{Dual lattice: Continuous time}

In this subsection, we define the first of two dual lattices and describe the correspondence between multiplicative harmonic functions and dual lattice vectors. Then we show the relationship between lengths of dual lattice vectors and eigenvalues of the continuous time sandpile chain. 

\begin{prop} 
\label{dual-lattice-1}
The map from $\Delta^{-1} \Z^{n-1} \to \HH$ given by
\begin{equation} \label{x-to-h}
x = (x_1,\ldots,x_{n-1}) \mapsto h(v_j) = \begin{cases} e^{2\pi i x_j} & \text{if } 1 \leq j \leq n-1, \\
1 & \text{if } j = n, \end{cases}
\end{equation}
is a surjective homomorphism with kernel $\Z^{n-1}$. Therefore $\HH \cong \Delta^{-1} \Z^{n-1} \big/ \Z^{n-1}$.
\end{prop}

We treat $x$ as a column vector; the notation $x = (x_1, \ldots, x_{n-1})$ is purely for convenience.

\begin{proof}
Suppose $x \in \R^{n-1}$ maps to $h$ by \eqref{x-to-h}. We claim that $h \in \HH$ if and only if $\Delta x \in \Z^{n-1}$. Indeed, $h$ satisfies the geometric mean value property at $v_j$ if and only if
\[
1 = h(v_j)^{\deg(v_j)} \prod_{v_k \sim v_j} h(v_k)^{-1} = e^{2\pi i (\Delta x)_j},
\]
where the second equality holds because $h(s) = 1$. Therefore $\Delta x \in \Z^{n-1}$ if and only if $h$ satisfies the geometric mean value property at every non-sink vertex (which implies the geometric mean value property at the sink).

It follows that \eqref{x-to-h} defines a surjective homomorphism from $\Delta^{-1} \Z^{n-1}$ to $\HH$. It is immediate from the definition that the kernel is $\Z^{n-1}$.
\end{proof}

\begin{remark}
\label{G-H-correspondence}
The isomorphism $\HH \cong \Delta^{-1} \Z^{n-1} \big/ \Z^{n-1}$ is the dual version of the isomorphism $\G \cong \Z^{n-1} \big/ \Delta \Z^{n-1}$. Although finite abelian groups are non-canonically isomorphic to their duals, in this case there is a natural map from $\Z^{n-1} \big/ \Delta \Z^{n-1}$ to $\Delta^{-1} \Z^{n-1} \big/ \Z^{n-1}$, namely multiplication by $\Delta^{-1}$. The corresponding map from $\G$ to $\HH$ can be described as follows: Given a sandpile configuration $\eta$ viewed as an element of $\Z^{n-1}$, let $x = \Delta^{-1} \eta$. Set $h(v_j) = e^{2\pi i x_j}$ for all $1 \leq j \leq n-1$ and $h(s) = 1$. If $\eta$ and $\eta'$ are equivalent configurations (that is, $\eta - \eta' \in \Delta \Z^{n-1}$), then the resulting functions $h,h'$ will be equal.

When $G$ is a directed graph, $\Delta$ is not symmetric, and $\HH \cong \Delta^{-1} \Z^{n-1} / \Z^{n-1}$ whereas $\G \cong \Z^{n-1} \big/ \Delta^\transpose \Z^{n-1}$; these groups are still isomorphic, but not naturally. 
\end{remark}

Fix $h \in \HH$, and choose $|x_j| \leq 1/2$ such that $h(v_j) = e^{2\pi i x_j}$ for all $1 \leq j \leq n$. (In particular, $x_n = 0$.) The next theorem quantifies the following simple idea: If the eigenvalue
\[
\lambda_h = \frac{1}{n} \sum_{j=1}^{n} e^{2\pi i x_j}
\]
is close to $1$, then the vector $(x_1,\ldots,x_{n-1})$ should be close to $0$; more precisely, the squared length $\|x\|_2^2$ is within a constant factor of $1 - \Re(\lambda_h)$. 

\begin{theorem}
\label{dual-length-1}
Given $h \in \HH$, choose $x \in \Delta^{-1} \Z^{n-1}$ of minimal Euclidean length such that $x \mapsto h$ via \eqref{x-to-h}. Then
\[
8\frac{\|x\|_2^2}{n} \leq 1 - \Re(\lambda_h) \leq 2\pi^2 \frac{\|x\|_2^2}{n}.
\]
\end{theorem}

\begin{proof}
We know that $|x_j| \leq 1/2$ for all $1 \leq j \leq n-1$, so $|2\pi x_j| \leq \pi$. We will use the inequality $1 - t^2\big/2 \leq \cos(t) \leq 1 - 2t^2 \big/ \pi^2$ for all $|t| \leq \pi$, where the upper bound is Lemma \ref{cos-upper-bound}. Since
\[
1 - \Re(\lambda_h) = \frac{1}{n} \sum_{j=1}^{n} [1 - \Re(h(v_j))] = \frac{1}{n} \sum_{j=1}^{n-1} [1 - \cos(2\pi x_j)],
\]
it follows that
\[
\frac{1}{n} \sum_{j=1}^{n-1} \frac{2}{\pi^2} (2\pi x_j)^2 \leq 1 - \Re(\lambda_h) \leq \frac{1}{n} \sum_{j=1}^{n-1} \frac{1}{2} (2\pi x_j)^2,
\]
which is the desired result.
\end{proof}

We can bound the continuous time spectral gap $\gamma_c$ by minimizing Theorem \ref{dual-length-1} over all $h \not\equiv 1$. Since $h \equiv 1$ corresponds to vectors $x \in \Z^{n-1}$, we immediately obtain the first statement in Theorem \ref{t.gap.intro}:

\begin{corollary}
\label{spec-gap-1}
The spectral gap of the continuous time sandpile chain satisfies
	\[ 8\frac{\|y\|_2^2}{n} \leq \gamma_c \leq 2\pi^2 \frac{\|y\|_2^2}{n}, \]
where $y$ is a vector of minimal Euclidean length in $(\Delta^{-1}\Z^{n-1}) \setminus \Z^{n-1}$.
\end{corollary}

There are cases in which the shortest vector in $\Delta^{-1} \Z^{n-1} \setminus \Z^{n-1}$ is much longer than the shortest nonzero vector in $\Delta^{-1} \Z^{n-1}$. For example, if $G$ is a cycle on $n$ vertices, then the length of the shortest vector in $\Delta^{-1} \Z^{n-1} \setminus \Z^{n-1}$ has order $\sqrt{n}$. But since $\Z^{n-1} \subseteq \Delta^{-1} \Z^{n-1}$, the lattice $\Delta^{-1} \Z^{n-1}$ contains vectors of length $1$.

\subsection{Dual lattice: Discrete time}
\label{Dual lattice: Discrete time}

To analyze the discrete time sandpile chain using $\HH \cong \Delta^{-1} \Z^{n-1} \big/ \Z^{n-1}$, we could argue as follows. The eigenvalue $\lambda_h$ is close to the unit circle when the values $h(v_j)$ are close to each other. This happens when the equivalence class in $\Delta^{-1} \Z^{n-1}$ associated with $h$ contains a vector $x$ close to the line $\{c \one : c \in \R \}$, where $\one$ denotes the all-ones vector. Instead of pursuing this approach, we start over with a different dual lattice construction that puts the sink on an equal footing with the other vertices. We will show that the analogue of $\Delta^{-1} \Z^{n-1}$ in this new construction is naturally associated with the discrete time sandpile chain in the same way that $\Delta^{-1} \Z^{n-1}$ is associated with the continuous time sandpile chain. Although the construction is somewhat involved, the reward---an easy proof of Theorem \ref{t.inverse.intro}---makes it worthwhile.

We first provide a sink-independent analogue of the group $\Z^{n-1} \big/ \Delta \Z^{n-1}$. Recall that $\overbar{\Delta}$ is the full $n$-dimensional Laplacian matrix of $G$. If $x \in \R^n$, then $\overbar{\Delta} x \in \R_0^n$. It can be checked that under the map from $\Z^{n-1}$ to $\Z_0^n$ that sends $(x_1,\ldots,x_{n-1}) \mapsto (x_1,\ldots,x_{n-1}, -\sum_{j=1}^{n-1} x_j)$, the image of $\Delta \Z^{n-1}$ is exactly $\overbar{\Delta} \Z^n$. Hence $\Z^{n-1} \big/ \Delta \Z^{n-1} \cong \Z_0^n \big/ \overbar{\Delta} \Z^n$. This isomorphism is used in \cite{HLMPPW} to prove that the sandpile group is independent of the choice of sink.

The analogue to $\Delta^{-1} \Z^{n-1}$ is the dual lattice $(\overbar{\Delta} \Z^n)^\ast = \{x \in \R_0^n \st \langle x,y \rangle \in \Z \text{ for all } y \in \overbar{\Delta} \Z^n\}$. We will see below that $(\overbar{\Delta} \Z^n)^\ast = \overbar{\Delta}^+ \Z_0^n$, where $\overbar{\Delta}^+$ is the \emph{Moore-Penrose pseudoinverse} of $\overbar{\Delta}$, an $n$-dimensional symmetric matrix which we now define.

Because the graph $G$ is connected, $\overbar{\Delta}$ has eigenvalue $0$ with multiplicity $1$, and its kernel is $\ker \overbar{\Delta} = \{ c \one : c \in \R \}$. The orthogonal subspace to $\ker \overbar{\Delta}$ is $\R_0^n$. Since $\overbar{\Delta}$ is symmetric, $\R_0^n$ is an invariant subspace for the map $x \mapsto \overbar{\Delta} x$. Indeed, this map is invertible when restricted to $\R_0^n$. Any vector in $\R^n$ can be written uniquely as $y + c\one$ for $y \in \R_0^n$ and $c \in \R$. The matrix $\overbar{\Delta}^+$ is defined by $\overbar{\Delta}^+ (y + c\one) = x$, where $x$ is the unique vector in $\R_0^n$ such that $\overbar{\Delta} x = y$. Thus the maps $x \mapsto \overbar{\Delta} x$ and $y \mapsto \overbar{\Delta}^+ y$ are inverses on $\R_0^n$. Also, $\ker \overbar{\Delta}^+ = \ker \overbar{\Delta}$, and the maps $x \mapsto \overbar{\Delta}^+ \overbar{\Delta} x$ and $y \mapsto \overbar{\Delta} \overbar{\Delta}^+ y$ on $\R^n$ are both the same as orthogonal projection onto $\R_0^n$. If the eigenvalues of $\overbar{\Delta}$ are $0 = \beta_0 < \beta_1 \leq \beta_2 \leq \cdots \leq \beta_{n-1}$, then the eigenvalues of $\overbar{\Delta}^+$ are $0 < \beta_{n-1}^{-1} \leq \beta_{n-2}^{-1} \leq \cdots \leq \beta_1^{-1}$.

To see why $(\overbar{\Delta} \Z^n)^\ast = \overbar{\Delta}^+ \Z_0^n$, suppose first that $x = \overbar{\Delta}^+ a$ for some $a \in \Z_0^n$. For any $b \in \Z^n$, $\langle x, \overbar{\Delta} b \rangle = \langle \overbar{\Delta} \overbar{\Delta}^+ a, b \rangle = \langle a,b \rangle \in \Z$. Thus $x \in (\overbar{\Delta} \Z^n)^\ast$. Conversely, if $x \in (\overbar{\Delta} \Z^n)^\ast$, then $x \in \R_0^n$ and $\langle \overbar{\Delta} x, e_j \rangle = \langle x, \overbar{\Delta} e_j \rangle \in \Z$ for all $1 \leq j \leq n$. This means $\overbar{\Delta} x \in \Z_0^n$, so $x \in \overbar{\Delta}^+ \Z_0^n$.

Here is the parallel to Proposition \ref{dual-lattice-1}.

\begin{prop} 
\label{dual-lattice-2}
The map from $\overbar{\Delta}^+ \Z_0^n \to \HH$ given by
\begin{equation} \label{x-to-h-2}
x = (x_1,\ldots,x_n) \mapsto h(v_j) = e^{2\pi i (x_j - x_n)}
\end{equation}
is a surjective homomorphism with kernel $\W^n$. Therefore $\HH \cong \overbar{\Delta}^+ \Z_0^n \big/ \W^n$.
\end{prop}

\begin{proof}
We first observe that the function $h(v_j) = e^{2\pi i y_j}$ is in $\HH$ if and only if the vector $y = (y_1,\ldots,y_n)$ satisfies $y_n \in \Z$ and $\overbar{\Delta} y \in \Z_0^n$. This is because the geometric mean value property at $v_j$ is the statement
\[
1 = h(v_j)^{\deg(v_j)} \prod_{v_k \sim v_j} h(v_k)^{-1} = e^{2\pi i (\overbar{\Delta} y)_j}.
\]
Note that $\overbar{\Delta} y \in \Z_0^n$ if and only if $\overbar{\Delta} y \in \Z^n$ since the columns of $\overbar{\Delta}$ sum to zero. The condition $y_n \in \Z$ is equivalent to $h(s) = 1$.

Let $x \in \overbar{\Delta}^+ \Z_0^n$ and define $h$ by \eqref{x-to-h-2}. We write $h(v_j) = e^{2\pi i y_j}$, where $y = x - x_n \one$. Then $y_n = 0$ and $\overbar{\Delta} y = \overbar{\Delta} x \in \overbar{\Delta} \overbar{\Delta}^+ \Z_0^n = \Z_0^n$, implying that $h \in \HH$.

It is clear that the map \eqref{x-to-h-2} is a homomorphism. To see that it is surjective, take any $h \in \HH$. Let $h(v_j) = e^{2\pi i y_j}$, so that $\overbar{\Delta} y \in \Z_0^n$. Let $c = \frac{1}{n}(y_1 + \cdots + y_n)$ and define $x = y - c \one$. Then $x \in \R_0^n$ and $\overbar{\Delta} x = \overbar{\Delta} y \in \Z_0^n$, meaning that $x = \overbar{\Delta}^+ \overbar{\Delta} x \in \overbar{\Delta}^+ \Z_0^n$. In addition, $x_j - x_n = y_j - y_n$ and so $e^{2\pi i (x_j - x_n)} = h(v_j)/h(v_n) = h(v_j)$ for each $j$. Thus $x \mapsto h$ via \eqref{x-to-h-2}.

For the kernel of \eqref{x-to-h-2}, first note that $\W^n = \overbar{\Delta}^+ \overbar{\Delta} \Z^n \subseteq \overbar{\Delta}^+ \Z_0^n$. A vector $x \in \overbar{\Delta}^+ \Z_0^n$ maps to $h \equiv 1$ via \eqref{x-to-h-2} if and only if $x_j - x_n \in \Z$ for all $j$. If $x \in \W^n$, we can write $x = y - c \one$ for some $y \in \Z^n$ and $c \in \R$. Then each $x_j - x_n = y_j - y_n \in \Z$. Conversely, if $x_j - x_n \in \Z$ for all $j$, let $y = x - x_n \one \in \Z^n$. Since $x \in \R_0^n$ is a translate of $y$ by a multiple of $\one$, $x$ is the orthogonal projection of $y$ onto $\R_0^n$, so $x \in \W^n$. This proves that $\W^n$ is the kernel of \eqref{x-to-h-2}.
\end{proof}

\begin{remark}
As in the other dual lattice construction, multiplication by $\overbar{\Delta}^+$ gives a natural map from $\Z_0^n \big/ \overbar{\Delta} \Z^n$ (which is isomorphic to $\G$) to $\overbar{\Delta}^+ \Z_0^n \big/ \overbar{\Delta}^+ \overbar{\Delta} \Z^n = \overbar{\Delta}^+ \Z_0^n \big/ \W^n$ (which is isomorphic to $\HH$). The corresponding map from $\G$ to $\HH$ is the same one
 described in the remark following Proposition \ref{dual-lattice-1}.
\end{remark}

Suppose $h \in \HH$ is given. Every $x$ that maps to $h$ via \eqref{x-to-h-2} arises by the following recipe: Choose $y \in \R^n$ such that $h(v_j) = e^{2\pi i y_j}$ for all $j$; set $c = \frac{1}{n} (y_1 + \cdots + y_n)$; and let $x = y - c \one$. Then
\[
\|x\|_2^2 = \sum_{j=1}^n (y_j - c)^2,
\]
so $x$ is short when the values $y_j$ are close to each other. This can only happen if the eigenvalue $\lambda_h = \frac{1}{n} \sum_{j=1}^n h(v_j)$ is close to the unit circle. The next theorem, which is parallel to Theorem \ref{dual-length-1}, makes this intuition precise.

\begin{theorem}
\label{dual-length-2}
Given $h \in \HH$, choose $x \in \overbar{\Delta}^+ \Z_0^n$ of minimal Euclidean length such that $x \mapsto h$ via \eqref{x-to-h-2}. Then
\[
8\frac{\|x\|_2^2}{n} \leq 1 - |\lambda_h| \leq 2\pi^2 \frac{\|x\|_2^2}{n}.
\]
\end{theorem}

\begin{proof}
The upper bound is straightforward: Since $\cos(t) \geq 1 - t^2/2$,
\[
\begin{split}
|\lambda_h| &= \left| \frac{1}{n} \sum_{j=1}^n e^{2\pi i (x_j - x_n)} \right| = \frac{1}{n} \left| \sum_{j=1}^n e^{2\pi i x_j} \right| \geq \frac{1}{n} \Re \left[ \sum_{j=1}^n e^{2\pi i x_j} \right] \\
&= \frac{1}{n} \sum_{j=1}^n \cos(2\pi x_j) \geq 1 - \frac{1}{n} \sum_{j=1}^n 2\pi^2 x_j^2 = 1 - 2\pi^2 \frac{\|x\|_2^2}{n}.
\end{split}
\]
For the lower bound, let $r = |\lambda_h|$ and write $\lambda_h = re^{2\pi i \theta}$. We will construct a vector $x_0 \in \overbar{\Delta}^+ \Z_0^n$ that maps to $h$ via \eqref{x-to-h-2} for which $8 \|x_0\|_2^2 \big/ n \leq 1-r$. Since $\|x\|_2 \leq \|x_0\|_2$, the result will follow.

Choose $y \in \R^n$ such that for all $j$, $h(v_j) = e^{2\pi i y_j}$ and $|y_j - \theta| \leq 1/2$. Then
\[
\lambda_h = re^{2\pi i \theta} = \frac{1}{n} \sum_{j=1}^n e^{2\pi i y_j}, \text{ which implies that } r = \frac{1}{n} \sum_{j=1}^n e^{2\pi i (y_j - \theta)}.
\]
Since $\cos(t) \leq 1 - 2t^2 \big/ \pi^2$ for all $|t| \leq \pi$ (by Lemma \ref{cos-upper-bound}),
\begin{equation}
\label{r-upper-bound}
|\lambda_h| = r = \Re(r) = \frac{1}{n} \sum_{j=1}^n \cos(2\pi (y_j - \theta)) \leq 1 - \frac{1}{n} \sum_{j=1}^n 8(y_j - \theta)^2.
\end{equation}
It follows that
\[
8 \frac{\|y - \theta \one \|_2^2}{n} \leq 1 - r.
\]
Let $x_0$ be the orthogonal projection of $y$ onto $\R_0^n$. Then $x_0 = y - c \one$, where $c = \frac{1}{n} (y_1 + \cdots + y_n)$. The surjectivity argument in the proof of Proposition \ref{dual-lattice-2} shows that $x_0 \in \overbar{\Delta}^+ \Z_0^n$ and that $x_0 \mapsto h$ via \eqref{x-to-h-2}. Since $x_0$ is also the orthogonal projection of $y - \theta \one$ onto $\R_0^n$, $\|x_0\|_2 \leq \|y - \theta \one\|_2$. Hence
\[
8 \frac{\|x\|_2^2}{n} \leq 8 \frac{\|x_0\|_2^2}{n} \leq 8 \frac{\|y - \theta \one \|_2^2}{n} \leq 1 - r. \qedhere
\]
\end{proof}

By minimizing Theorem \ref{dual-length-2} over all $h \not\equiv 1$, noting that $h \equiv 1$ corresponds to vectors $x \in \W^n$, we obtain the second statement in Theorem \ref{t.gap.intro}:

\begin{corollary}
\label{spec-gap-2}
The spectral gap $\gamma_d$ of the discrete time sandpile chain satisfies
	\[ 8\frac{\|z\|_2^2}{n} \leq \gamma_d \leq 2 \pi^2 \frac{\|z\|_2^2}{n}, \]
where $z$ is a vector of minimal Euclidean length in $(\overbar{\Delta}^+ \Z_0^n) \setminus \W^n$.
\end{corollary}

As with the remark following Corollary \ref{spec-gap-1}, the cycle on $n$ vertices provides an example where the shortest vector in $\overbar{\Delta}^+ \Z_0^n \setminus \W^n$ has length of order $\sqrt{n}$, while the shortest nonzero vector in $\overbar{\Delta}^+ \Z_0^n$ has length slightly less than $1$.

\subsection{An inverse relationship}
\label{An inverse relationship}

Let $0 = \beta_0 < \beta_1 \leq \cdots \leq \beta_{n-1}$ be the eigenvalues of the full Laplacian matrix $\overbar{\Delta}$. The least positive eigenvalue $\beta_1$ is called the spectral gap of the graph $G$, or sometimes the \emph{algebraic connectivity} of $G$. When $G$ is a $d$-regular graph, $\beta_1 / d$ is the spectral gap of the simple random walk on $G$.\footnote{The second-largest eigenvalue of the transition matrix is $1 - \beta_1 / d$. This disagrees with our earlier definition of the spectral gap as the minimum \emph{absolute} distance between a nontrivial eigenvalue and the unit circle, but it accords with standard usage for reversible Markov chains.} Larger values of $\beta_1$ mean that $G$ is `more connected,' and in the $d$-regular case that the simple random walk on $G$ mixes faster (setting aside issues of periodicity).

In this subsection we prove Theorem \ref{t.inverse.intro}, which says that the spectral gap of the discrete time sandpile chain satisfies
\[
\gamma_d \leq \frac{4 \pi^2}{\beta_1^2 n}.
\]
Surprisingly, large values of $\beta_1$ cause the sandpile chain to mix slowly! This inverse relationship is reminiscent of the bound of Bj\"{o}rner, Lov\'{a}sz and Shor \cite[Theorem 4.2]{BLS91} on the number of topplings until a configuration stabilizes, which has a factor of $\beta_1$ in the denominator and is also proved using the pseudoinverse $\overbar{\Delta}^+$. 

\begin{proof}[Proof of Theorem \ref{t.inverse.intro}]
By Corollary \ref{spec-gap-2}, $\gamma_d$ is determined up to a constant factor by the length of the shortest vector in the dual lattice $\overbar{\Delta}^+ \Z_0^n$ that is not also in the sublattice $\W^n$. The eigenvalues of $\overbar{\Delta}^+$ are $0 < \beta_{n-1}^{-1} \leq \beta_{n-2}^{-1} \leq \cdots \leq \beta_1^{-1}$. Therefore a lower bound on $\beta_1$ gives an upper bound on the operator norm of $\overbar{\Delta}^+$.

Since $\HH \cong \overbar{\Delta}^+ \Z_0^n \big/ \W^n$ by Proposition \ref{dual-lattice-2}, we may assume that $\W^n$ is a proper sublattice of $\overbar{\Delta}^+ \Z_0^n$ (otherwise we would have $|\HH| = 1$ and $\gamma_d$ itself would be meaningless). The lattice $\Z_0^n$ is generated by the vectors $e_i-e_j$ with $1 \leq i < j \leq n$. Thus we can choose $i,j$ such that $\overbar{\Delta}^+ (e_i-e_j) \notin \W^n$. We have
\[
\|\overbar{\Delta}^+ (e_i-e_j)\|_2^2 \leq \frac{\|e_i-e_j\|_2^2}{\beta_1^2} = \frac{2}{\beta_1^2}.
\]
By Corollary \ref{spec-gap-2},
\[
\gamma_d \leq \frac{2\pi^2}{n} \|\overbar{\Delta}^+ (e_i-e_j)\|_2^2 \leq \frac{4 \pi^2}{\beta_1^2 n}. \qedhere
\]
\end{proof}

Later in this subsection we will consider a class of graphs for which the bound in Theorem \ref{t.inverse.intro} is accurate to a constant factor. In many cases, though, it can be far from sharp. For example, when $G$ is the torus $\Z_m \times \Z_m$ (so that $n = m^2$), $\beta_1$ has order $1/n$.

The analogue of Theorem \ref{t.inverse.intro} for the continuous time spectral gap is harder to prove because we must use $\Delta^{-1}$ instead of $\overbar{\Delta}^+$. If $0 < \rho_1 \leq \rho_2 \leq \cdots \leq \rho_{n-1}$ are the eigenvalues of $\Delta$, then the Cauchy interlacing theorem \cite[Ch.\ 4]{HJ90} implies that $\beta_{j-1} \leq \rho_j \leq \beta_j$ for all $1 \leq j \leq n-1$. Therefore the eigenvalues $\rho_{n-1}^{-1}, \ldots, \rho_2^{-1}$ of $\Delta^{-1}$ are all bounded above by $\beta_1^{-1}$, but $\rho_1^{-1}$ may be much larger than $\beta_1^{-1}$. A parallel argument to the proof of Theorem \ref{t.inverse.intro} would break down because of this large eigenvalue. Nevertheless, with considerable effort we can show the following bound. We omit the proof for reasons of space.

\begin{theorem}
\label{inverse-cont}
Suppose the graph $G$ has no vertices of degree $1$. The spectral gap of the continuous time sandpile chain satisfies
\[
\gamma_c \leq \frac{4\pi^2 + O(1/n)}{\beta_1^2 n} \leq \frac{10\pi^2}{\beta_1^2 n}.
\]
\end{theorem}

Combining Theorems \ref{lazy_gap_bd} and \ref{t.inverse.intro} gives both upper and lower bounds on the discrete time spectral gap:
\[
\frac{8}{d_\ast^2 n} \leq \gamma_d \leq \frac{4\pi^2}{\beta_1^2 n},
\]
where $d_\ast$ is the second-largest vertex degree. We now introduce a class of graphs for which the upper and lower bounds have the same order. These are, in our language, 
the expander graphs with bounded degree ratio.

Loosely, the graph $G$ is an expander if it is sparse and if the simple random walk on $G$ mixes quickly. Expanders have many applications in pure and applied mathematics as well as theoretical computer science. See the surveys \cite{L12,HLW06} for more information.

For the purposes of this paper, we will not need the requirement of sparsity. Let $K$ be the transition matrix for the simple random walk on $G$, and let $L = I - K$ be the \emph{random walk Laplacian}. Since $K$ is reversible with respect to the stationary distribution $\pi(v) = \deg(v) \big/ \sum_{w \in V} \deg(w)$, $L$ has real eigenvalues $0 = \theta_0 < \theta_1 \leq \theta_2 \leq \cdots \leq \theta_{n-1}$. The value $\theta_1$ is called the spectral gap of $L$. When $G$ is $d$-regular, $\overbar{\Delta} = dL$ and so $\beta_1 = d \theta_1$.

\begin{definition} 
\label{expander-def}
Fix $\alpha > 0$, and let $L$ be the random walk Laplacian matrix for the graph $G$. We say $G$ is an \emph{$\alpha$-expander} if the spectral gap of $L$ satisfies $\theta_1 \geq \alpha$.
\end{definition}

The following lemma, whose proof is postponed to the end of this section, gives a relationship between $\beta_1$ and $\theta_1$ when $G$ is not regular.

\begin{lemma} 
\label{spec-gap-degrees}
Let $G$ have minimum and maximum vertex degrees $\dmin$ and $\dmax$. The eigenvalues $0 = \beta_0 < \beta_1 \leq \cdots \leq \beta_{n-1}$ of $\overbar{\Delta}$ and $0 = \theta_0 < \theta_1 \leq \cdots \leq \theta_{n-1}$ of $L$ are related by
\[
\dmin \theta_i \leq \beta_i \leq \dmax \theta_i, \qquad 0 \leq i \leq n-1.
\]
\end{lemma}

In particular, $\beta_1 \geq \dmin \theta_1$. Thus if $G$ is an $\alpha$-expander and the \emph{degree ratio} $\dmax / \dmin$ is bounded by a constant $R$, then $\beta_1 \geq (\alpha/R) \dmax$.

\begin{prop} 
\label{expander}
Suppose that for fixed $\alpha, R > 0$, the graph $G$ is an $\alpha$-expander and $\dmax / \dmin \leq R$. Then
\[
\frac{8}{\dmax^2 n} \leq \gamma_d \leq \frac{4\pi^2 R^2 \big/ \alpha^2}{\dmax^2 n}.
\]
\end{prop}

\begin{proof}
The lower bound is Theorem \ref{lazy_gap_bd}. The upper bound is Theorem \ref{t.inverse.intro} combined with Lemma \ref{spec-gap-degrees}, as discussed above.
\end{proof}

For expanders with bounded degree ratio, Proposition \ref{expander} determines the spectral gap $\gamma_d$ of the discrete time sandpile chain up to a constant factor. 

Note that we could replace $\dmax$ with $d_\ast$ in the statement of Proposition \ref{expander} to get a slightly stronger result with the exact same proof. This is useful in graphs where a single vertex has much higher degree than all the others.

We can prove an analogue of Proposition \ref{expander} for the continuous time spectral gap $\gamma_c$. The proof uses Theorem \ref{inverse-cont} when $\dmin \geq 2$ and a separate argument when $\dmin = 1$. The lower bound is the same, since $\gamma_c \geq \gamma_d$. In the upper bound, the numerator $4\pi^2 R^2 / \alpha^2$ is replaced by $1300 R^3 / \alpha^2$.

Proposition \ref{expander} allows us to determine the order of the relaxation time for the sandpile chain 
on some well-known random graphs.

\begin{prop} 
\label{d-regular}
Fix $d\geq 3$, and let $(G_n)$ be a sequence of random $d$-regular graphs with $|V(G_n)| = n \to \infty$.\footnote{More precisely, assume that for each $n$ in the sequence, the set of simple $d$-regular graphs on $n$ vertices is nonempty. Choose $G_n$ uniformly from this set, independently for each $n$.} The spectral gaps $\gamma_d(n)$ for the discrete time sandpile chain satisfy $\gamma_d = \Theta(1/n)$ almost surely.
\end{prop}

\begin{prop} 
\label{Erdos-Renyi}
Let $p = p(n)$ satisfy $0 \leq p \leq 1$ and $np / \log(n) \to \infty$ as $n \to \infty$. Let $G_n = G(n,p)$ be independently chosen Erd\H os-R\'enyi random graphs on $n$ vertices with edge probability $p$. The spectral gaps $\gamma_d(n)$ for the discrete time sandpile chain satisfy $\gamma_d = \Theta(1/n^3 p^2)$ almost surely.
\end{prop}

Both these propositions are simple applications of Proposition \ref{expander}. For Proposition \ref{d-regular}, all we need is the well-known fact that random $d$-regular graphs are expanders. A famous result of Friedman \cite{F08} says that for every $\e > 0$, the random walk spectral gap satisfies $\theta_1 \geq 1 - 2\sqrt{d-1} \big/ d - \e$ asymptotically almost surely.

For Proposition \ref{Erdos-Renyi}, we first note that in our regime, $\dmin / np$ and $\dmax / np$ converge to $1$ almost surely as $n \to \infty$. (Use Bernstein's Inequality to show that the degree of each vertex concentrates around its mean $(n-1)p$, then take a union bound over all the vertices.) To prove expansion, Theorem 2 of \cite{CR11} implies that because $np / \log(n) \to \infty$, almost surely $\theta_1 \geq 1 - o(1)$.

We conclude with the proof of Lemma \ref{spec-gap-degrees}.

\begin{proof}[Proof of Lemma \ref{spec-gap-degrees}]
Define $Z = \sum_{j=1}^n \deg(v_j)$. The stationary distribution $\pi$ of the simple random walk on $G$ is $\pi(v_j) = \deg(v_j) \big/ Z$. Consider these two inner products on the space $\R^V = \{f: V \to \R \}$:
\[
\langle f,g \rangle_\pi = \sum_{j=1}^n f(v_j)g(v_j)\pi(v_j), \qquad \langle f,g \rangle = \sum_{j=1}^n f(v_j)g(v_j).
\]
The matrix $\overbar{\Delta}$ is self-adjoint with respect to $\langle \cdot,\cdot \rangle$, while $L$ is self-adjoint with respect to $\langle \cdot,\cdot \rangle_\pi$. We have $\overbar{\Delta} = DL$ where $D$ is the diagonal matrix with entries $D(j,j) = \deg(v_j)$. The two inner products are related by
\[
\langle f,Dg \rangle = \sum_{j=1}^n f(v_j)\deg(v_j) g(v_j) = Z\langle f,g \rangle_\pi.
\]
In addition,
\[
\frac{\dmin}{Z} \langle f,f \rangle \leq \sum_{j=1}^n \frac{\deg(v_j)}{Z} f(v_j)^2 = \langle f,f \rangle_\pi \leq \frac{\dmax}{Z} \langle f,f \rangle.
\]
It follows that for $f \not\equiv 0$,
\begin{multline} \label{Rayleigh-ineq}
\dmin \left( \frac{ \langle f, Lf \rangle_\pi}{\langle f,f \rangle_\pi} \right) = \frac{Z \langle f, Lf \rangle_\pi}{(Z\big/\dmin)\langle f,f \rangle_\pi} \leq \frac{\langle f, DLf \rangle}{\langle f,f \rangle} \\
\leq \frac{Z \langle f, Lf \rangle_\pi}{(Z\big/\dmax)\langle f,f \rangle_\pi} = \dmax \left( \frac{ \langle f, Lf \rangle_\pi}{\langle f,f \rangle_\pi} \right).
\end{multline}
Since $L$ is self-adjoint with respect to $\langle \cdot,\cdot \rangle_\pi$ and $\overbar{\Delta} = DL$ is symmetric, the variational characterization of eigenvalues (see \cite{HJ90}) says that for each $0 \leq i \leq n-1$,
\begin{equation} \label{theta-variational}
\beta_i = \min_{\substack{U\subseteq\R^V \\ \dim(U) = i+1}} \max_{f\in U, f \not\equiv 0} \frac{\langle f, \overbar{\Delta} f \rangle}{\langle f,f \rangle}.
\end{equation}
Here the minimum is taken over all linear subspaces $U\subseteq\R^V$ of dimension $i+1$. Likewise,
\begin{equation} \label{beta-variational}
\theta_i = \min_{\substack{U\subseteq\R^V \\ \dim(U) = i+1}} \max_{f\in U, f \not\equiv 0} \frac{ \langle f, Lf \rangle_\pi}{\langle f,f \rangle_\pi}.
\end{equation}
Comparing the characterizations \eqref{theta-variational} and \eqref{beta-variational} and using the upper and lower bounds in \eqref{Rayleigh-ineq}, we conclude that
\[
\dmin \theta_i \leq \beta_i \leq \dmax \theta_i. \qedhere
\]
\end{proof}

\section{The Smoothing Parameter}
\label{The Smoothing Parameter}

This section discusses the relationship between the mixing time of the sandpile chain and a lattice invariant called the \emph{smoothing parameter}, which was introduced by Micciancio and Regev \cite{MR07} in the context of lattice-based cryptography. Our main result, Theorem \ref{t.smoothing.revised}, is a slightly stronger version of Theorem \ref{t.smoothing.intro}. It provides an upper bound on the $L^2$ mixing time in terms of the smoothing parameter for the Laplacian lattice.

The starting point of the proof is the relationship between lengths of dual lattice vectors and eigenvalues of the sandpile chain developed in Sections \ref{Dual lattice: Continuous time} and \ref{Dual lattice: Discrete time}. Previously we focused on the consequences for the spectral gap. Here we consider all the eigenvalues in order to control the $L^2$ distance from stationarity.

With Theorem \ref{t.smoothing.revised} in hand, we can take advantage of existing results that relate the smoothing parameter to other lattice invariants. By this approach, we show that if $d_\ast$ is the second-highest vertex degree in $G$, then the mixing times for both the discrete and continuous time sandpile chains are at most $\frac{1}{16}d_\ast^2 n \log n$. The precise statement is Theorem \ref{smoothing-mixing-time} below.

We already knew from Theorem \ref{lazy_gap_bd} that the relaxation time is at most $\frac{1}{8}d_\ast^2 n$, so (by Proposition \ref{TV_eigval}) the mixing time is at most $\frac{1}{16}(d_\ast^2 \log d_\ast) n^2$. The new bound of $\frac{1}{16}d_\ast^2 n \log n$ is a significant improvement. When $d_\ast \gg \sqrt{n}$, Theorem \ref{cutoff-upper-bd} further improves the leading constant from $\frac{1}{16}$ to $\frac{1}{4\pi^2}$. We will see in Section \ref{Complete graph} that this bound is sharp for the discrete time sandpile chain on the complete graph.

Let us begin by recalling the definition of the smoothing parameter from Section \ref{Introduction}. Let $\Lambda$ be a lattice in $\R^m$, and let $V = \Span(\Lambda) \subseteq \R^m$. For $s > 0$, the function
\[
f_{\Lambda}(s) = \sum_{x \in \Lambda^\ast \setminus \{\zero\}} e^{-\pi s^2 \|x\|_2^2}
\]
is continuous and strictly decreasing, with a limit of $\infty$ as $s \to 0$ and a limit of $0$ as $s \to \infty$. For $\e > 0$, the smoothing parameter of $\Lambda$ is defined as $\eta_\e(\Lambda) := f_\Lambda^{-1}(\e)$. Therefore,
\begin{equation} \label{smoothing-def}
\sum_{x \in \Lambda^\ast \setminus \{\zero\}} e^{-\pi \eta_\e^2(\Lambda) \|x\|_2^2} = \e.
\end{equation}

In \cite{MR07} the smoothing parameter was defined only for lattices of full rank. The definition given above is a natural generalization. Indeed, suppose $\Lambda$ has rank $k < m$. Fix a vector space isomorphism $\Phi : V \to \R^k$ that preserves the standard inner product. We see that $\Phi(\Lambda^\ast) = (\Phi(\Lambda))^\ast$, so $f_\Lambda(s) = f_{\Phi(\Lambda)}(s)$ and $\eta_\e(\Lambda) = \eta_\e(\Phi(\Lambda))$. Results that relate the smoothing parameter to other values invariant under $\Phi$, such as lengths of vectors, carry over without change to the general setting.

The reason for the name `smoothing parameter' is an alternative characterization that we will not use in our proofs. Let $\Lambda \subseteq \R^k$ be a full-rank lattice. Given $c \in \R^k$ and $r > 0$, push forward the density for a Gaussian on $\R^k$ with mean $c$ and covariance matrix $\frac{r^2}{2\pi} \mathrm{Id}$ by the quotient map $\R^k \to \mathcal{D} := \R^k / \Lambda$. It is shown in \cite{MR07} that the $L^\infty$ distance between this density and the uniform density on $\mathcal{D}$ is independent of $c$, and equals $\e$ precisely when $r = \eta_\e(\Lambda)$. Thus $\eta_\e(\Lambda)$ is the amount of scaling necessary for a Gaussian to be nearly uniformly distributed over $\mathcal{D}$, with an $L^\infty$ error of $\e$.

\begin{theorem}
\label{t.smoothing.revised}
Let $H_{id}^t$ and $P_{id}^t$ denote the distributions at time $t$ of the continuous and discrete time sandpile chains, respectively, started from the identity, and let $U$ denote the uniform distribution on $\G$. For any $\e > 0$,
\begin{gather}
\|H_{id}^t - U \|_2^2 \leq \e \enspace \text{\em for all } t \geq \frac{\pi}{16} n \cdot \eta_\e^2(\Delta \Z^{n-1}), \label{t-lower-bound-1} \\[3pt]
\max\left\{ \|H_{id}^t - U \|_2^2, \, \|P_{id}^t - U \|_2^2 \right\} \leq \e \enspace \text{\em for all } t \geq \frac{\pi}{16}n \cdot \eta_\e^2(\overbar{\Delta} \Z^n). \label{t-lower-bound-2}
\end{gather}
\end{theorem}

\noindent Since $\|H_\sigma^t - U\|_2^2 = \|H_{id}^t - U\|_2^2$ and $\|P_\sigma^t - U\|_2^2 = \|P_{id}^t - U\|_2^2$ for any $\sigma \in \G$, Theorem \ref{t.smoothing.revised} implies Theorem \ref{t.smoothing.intro}. The only difference is that the bound \eqref{t-lower-bound-2} applies to both the continuous and discrete time sandpile chains.

\begin{proof}[Proof of Theorem \ref{t.smoothing.revised}]
The proofs of \eqref{t-lower-bound-1} and \eqref{t-lower-bound-2} are almost identical. For \eqref{t-lower-bound-1}, start with the formula
\[
\|H_{id}^t - U \|_2^2 = \sum_{h \in \HH \setminus \{1\}} e^{-2t(1 - \Re(\lambda_h))}.
\]
In Section \ref{Dual lattice: Continuous time} we saw that each $h \in \HH$ corresponds to an equivalence class of vectors in the dual lattice $\Delta^{-1} \Z^{n-1}$. Let $x_h$ be a member of the equivalence class corresponding to $h$, chosen to have minimal Euclidean length. By Theorem \ref{dual-length-1},
\[
1 - \Re(\lambda_h) \geq 8 \frac{\|x_h\|_2^2}{n}.
\]
Let $W = \{x_h : h \in \HH \} \subseteq \Delta^{-1} \Z^{n-1}$. Since $h \equiv 1$ corresponds to $x_h = \zero$,
\[
\|H_{id}^t - U \|_2^2 \leq \sum_{x \in W \setminus \{\zero\}} e^{-\frac{16t}{n} \|x\|_2^2} \leq \sum_{x \in \Delta^{-1} \Z^{n-1} \setminus \{\zero\}} e^{-\frac{16t}{n} \|x\|_2^2}.
\]
Now using the lower bound on $t$ in \eqref{t-lower-bound-1} along with \eqref{smoothing-def},
\[
\sum_{x \in \Delta^{-1} \Z^{n-1} \setminus \{\zero\}} e^{-\frac{16t}{n} \|x\|_2^2} \leq \sum_{x \in \Delta^{-1} \Z^{n-1} \setminus \{\zero\}} e^{-\pi \eta_\e^2(\Delta \Z^{n-1}) \|x\|_2^2} = \e.
\]
This proves \eqref{t-lower-bound-1}.

To show \eqref{t-lower-bound-2}, we have
\begin{gather*}
\|H_{id}^t - U\|_2^2 = \sum_{h \in \HH \setminus \{1\}} e^{-2t(1 - \Re(\lambda_h))} \leq \sum_{h \in \HH \setminus \{1\}} e^{-2t(1 - |\lambda_h|)}, \\
\|P_{id}^t - U \|_2^2 = \sum_{h \in \HH \setminus \{1\}} |\lambda_h|^{2t} = \sum_{h \in \HH \setminus \{1\}} e^{2t \log |\lambda_h|} \leq \sum_{h \in \HH \setminus \{1\}} e^{-2t(1 - |\lambda_h|)}.
\end{gather*}
From there the proof is the same as above, using the dual lattice $\overbar{\Delta}^+ \Z_0^n$ in place of $\Delta^{-1} \Z^{n-1}$ and Theorem \ref{dual-length-2} in place of Theorem \ref{dual-length-1}.
\end{proof}

Theorem \ref{smoothing-optimized} will show that the leading constant $\pi / 16$ in Theorem \ref{t.smoothing.revised} can be improved under certain conditions.

To apply Theorem \ref{t.smoothing.revised}, we need to bound the smoothing parameter in terms of other lattice invariants. There are several such bounds in the literature. We will use the following result of \cite{MR07}.

\begin{lemma}[\cite{MR07}, Lemma 3.3] \label{MR3.3}
For any lattice $\Lambda$ of rank $k$ and any $\e > 0$,
\[
\eta_\e(\Lambda) \leq \sqrt{ \frac{\log(2k(1 + 1/\e))}{\pi} } \cdot \lambda_k(\Lambda),
\]
where $\lambda_k(\Lambda)$ is the least real number $r$ such that the closed Euclidean ball of radius $r$ about the origin contains at least $k$ linearly independent vectors in $\Lambda$.
\end{lemma}

(In \cite{MR07} the lemma is stated for lattices of full rank. It extends to the general case by the previous remark about the isomorphism $\Phi$.)

If we combine Theorem \ref{t.smoothing.revised} with Lemma \ref{MR3.3} and bound $\lambda_k(\Lambda)$ using the entries of the Laplacian matrix, we obtain the following bound on mixing time.

\begin{theorem} \label{smoothing-mixing-time}
Let $d_\ast$ be the second-highest vertex degree in $G$. For any $\e > 0$,
\[
\max\left\{ \|H_{id}^t - U\|_2^2, \, \|P_{id}^t - U \|_2^2 \right\} \leq \e \enspace \text{\em for all } t \geq \frac{1}{16} (d_\ast^2 + d_\ast) n \log[2(n-1)(1 + 1/\e)].
\]
\end{theorem}

\begin{proof}
To bound $\lambda_{n-1}(\overbar{\Delta} \Z^n)$, observe that $\overbar{\Delta} \Z^n$ is generated by any $n-1$ columns of $\overbar{\Delta}$. We omit the column corresponding to the highest-degree vertex. All the other columns satisfy $\|x\|_2^2 \leq d_\ast^2 + d_\ast$, implying that $\lambda_{n-1}^2(\overbar{\Delta} \Z^n) \leq d_\ast^2 + d_\ast$.

Lemma \ref{MR3.3} yields
\begin{equation}
\eta_\e^2(\overbar{\Delta} \Z^n) \leq \frac{1}{\pi} (d_\ast^2 + d_\ast) \log[ 2(n-1)(1 + 1/\e)]. \label{smoothing-discrete-bound}
\end{equation}
The proof is finished by plugging this bound into \eqref{t-lower-bound-2}.
\end{proof}

\old{
\begin{remark}
\label{highest-nonsink-degree}
By using \eqref{t-lower-bound-1} in place of \eqref{t-lower-bound-2}, we could prove that
\[
\|H_{id}^t - U \|_2^2 \leq \e \enspace \text{for all } t \geq \frac{1}{16} (d_\circledast^2 + d_\circledast) n \log[2(n-1)(1 + 1/\e)],
\]
where $d_\circledast$ is the highest degree of a non-sink vertex. Since $d_\circledast \geq d_\ast$, this bound cannot be better than what was shown above. However, this reasoning says nothing about whether \eqref{t-lower-bound-1} or \eqref{t-lower-bound-2} gives a tighter bound on $\|H_{id}^t - U\|_2^2$ in general.
\end{remark}
}

Theorem \ref{smoothing-mixing-time} shows that every sandpile chain on a graph $G$ with second-highest vertex degree $d_\ast$ has $L^2$ mixing time at most $\frac{1}{16} d_\ast^2 n \log n$ plus lower-order terms. When $G$ is the complete graph, we will see below that the order $d_\ast^2 n \log n$ is correct but the constant $\frac{1}{16}$ is not sharp. By contrast, when $G$ is the cycle graph or the `triangle with tail' (see Section \ref{Discrete vs Continuous}), the order $d_\ast^2 n \log n$ is not sharp. 

We now state a modified version of Theorem \ref{t.smoothing.revised} which improves the leading constant from $\pi/16$ to $1/4\pi$ at the cost of an extra term. As usual, $\eta_\e$ denotes the smoothing parameter, and $|\G|$ the order of the sandpile group. 

\begin{theorem}
\label{smoothing-optimized}
For any $\e > 0$,
\begin{gather*}
\|H_{id}^t - U\|_2^2 \leq 2\e \enspace \text{\em for all } t \geq \frac{1}{4\pi} n \cdot \eta_\e^2(\Delta \Z^{n-1}) + \frac{\pi^2}{48} n \log\left( \frac{|\G|}{\e} \right), \\
\max\left\{ \|H_{id}^t - U\|_2^2, \, \|P_{id}^t - U\|_2^2 \right\} \leq 2\e \enspace \text{\em for all } t \geq \frac{1}{4\pi} n \cdot \eta_\e^2(\overbar{\Delta} \Z^n) + \frac{\pi^2}{48} n \log\left( \frac{|\G|}{\e} \right).
\end{gather*}
\end{theorem}

When we combine the second statement of Theorem \ref{smoothing-optimized} with \eqref{smoothing-discrete-bound} and the inequality $\log(|\G|) \leq (n-1) \log(d_\ast)$, we obtain the following result.

\begin{theorem}
\label{cutoff-upper-bd}
Let $d_\ast$ be as in Theorem \ref{smoothing-mixing-time}. For any $c \geq 1$,
\[
\max\left\{ \|H_{id}^t - U\|_2^2, \|P_{id}^t - U\|_2^2 \right\} \leq e^{-c} \text{ \em for all } t \geq \frac{1}{4\pi^2} d_\ast^2 n \log(n) + \frac{n^2}{4} \log(d_\ast) + c d_\ast^2 n.
\]
\end{theorem}

We will show in Section \ref{Complete graph} that the leading term $\frac{1}{4\pi^2} d_\ast^2 n \log n$ is sharp for the discrete time sandpile chain on the complete graph. In general, Theorem \ref{cutoff-upper-bd} is better than Theorem \ref{smoothing-mixing-time} when $d_\ast \gg \sqrt{n}$ and worse when $d_\ast \ll \sqrt{n}$.

\begin{proof}[Proof of Theorem \ref{smoothing-optimized}]
We start by proving the second statement of the theorem.

Fix $b > 0$; we will specify its value later. We partition the set $\HH$ of multiplicative harmonic functions into two subsets $\HH_1, \HH_2$ as follows. Given $h \in \HH$, write $\lambda_h = re^{2\pi i \theta}$ with $r \geq 0$. (If $r = 0$ then let $\theta = 0$.) Choose $y \in \R^n$ such that for all $j$, $h(v_j) = e^{2\pi i y_j}$ and $|y_j - \theta| \leq 1/2$. As in the proof of Theorem \ref{dual-length-2},
\[
|\lambda_h| = \frac{1}{n} \sum_{j=1}^n \cos(2\pi(y_j - \theta)).
\]
Now define
\begin{align*}
\HH_1 &= \{h \in \HH \st \max_{1 \leq j \leq n} 2\pi|y_j - \theta| \leq b \}, \\
\HH_2 &= \{h \in \HH \st \max_{1 \leq j \leq n} 2\pi|y_j - \theta| > b \}.
\end{align*}
Note that $\HH_2$ is empty if $b \geq \pi$. The reason for the partition is that when $h \in \HH_1$, the cosine approximation from Lemma \ref{cos-upper-bound} is improved; and when $h \in \HH_2$, the existence of $j$ for which $2\pi |y_j - \theta| > b$ makes $|\lambda_h|$ relatively far from $1$.

For $h \in \HH_1$, Lemma \ref{cos-upper-bound} yields
\[
|\lambda_h| \leq \frac{1}{n} \sum_{j=1}^n \left[ 1 - \frac{1 - \cos(b)}{b^2} \cdot 4\pi^2 (y_j - \theta)^2 \right] \leq 1 - \frac{4\pi^2}{n} \left( \frac{1}{2} - \frac{b^2}{24} \right) \sum_{j=1}^n (y_j - \theta)^2,
\]
using that $1 - \cos(b) \geq b^2/2 - b^4/24$ in the second inequality. Compare with \eqref{r-upper-bound}. Let $x_h \in \overbar{\Delta}^+ \Z_0^n$ be a dual lattice vector of minimal Euclidean length corresponding to $h$. Arguing as in the proof of Theorem \ref{dual-length-2}, we conclude that
\begin{equation} \label{H1-lower-bound}
1 - |\lambda_h| \geq \frac{4\pi^2}{n} \left( \frac{1}{2} - \frac{b^2}{24} \right) \|x_h\|_2^2.
\end{equation}

For $h \in \HH_2$, there exists $j$ such that $b < 2\pi|y_j - \theta| \leq \pi$, which gives the bound
\[
|\lambda_h| \leq 1 - \frac{1}{n} +\frac{1}{n} \cos(b).
\]
Since $b < \pi$, Lemma \ref{cos-upper-bound} says that $\cos(b) \leq 1 - 2b^2 / \pi^2$, so
\begin{equation} \label{H2-lower-bound}
1 - |\lambda_h| \geq \frac{1}{n} \cdot \frac{2b^2}{\pi^2}.
\end{equation}

We are now ready to bound $\max\left\{ \|H_{id}^t - U\|_2^2, \, \|P_{id}^t - U\|_2^2 \right\}$. To begin,
\[
\begin{split}
\max\left\{ \|H_{id}^t - U\|_2^2, \, \|P_{id}^t - U\|_2^2 \right\} &\leq \sum_{h \in \HH \setminus \{1\}} e^{-2t(1 - |\lambda_h|)} \\
&= \sum_{h \in \HH_1 \setminus \{1\}} e^{-2t(1 - |\lambda_h|)} + \sum_{h \in \HH_2} e^{-2t(1 - |\lambda_h|)}.
\end{split}
\]
For the first sum, let $W_1 = \{x_h \st h \in \HH_1\} \subseteq \overbar{\Delta}^+ \Z_0^n$. By \eqref{H1-lower-bound},
\[
\sum_{h \in \HH_1 \setminus \{1\}} e^{-2t(1 - |\lambda_h|)} \leq \sum_{x \in W_1 \setminus \{\zero\}} e^{-2t \cdot \frac{4\pi^2}{n} \left( \frac{1}{2} - \frac{b^2}{24} \right) \|x\|_2^2}.
\]
If we choose $b$ so that
\begin{equation}
\label{b-condition-1}
2t \cdot \frac{4\pi^2}{n} \left( \frac{1}{2} - \frac{b^2}{24} \right) \geq \pi \eta_\e^2(\overbar{\Delta} \Z^n),
\end{equation}
then this sum will be at most $\e$ by the definition of the smoothing parameter.

The second sum is zero when $b \geq \pi$. When $b < \pi$, we use \eqref{H2-lower-bound} and the inequality $|\HH_2| \leq |\G|$ to obtain
\[
\sum_{h \in \HH_2} e^{-2t(1 - |\lambda_h|)} \leq |\G| e^{-\frac{4tb^2}{\pi^2 n}} = e^{\log(|\G|) - \frac{4tb^2}{\pi^2 n}}.
\]
This quantity will be at most $\e$ provided that
\begin{equation}
\label{b-condition-2}
\log |\G| - \frac{4tb^2}{\pi^2 n} \leq \log \e.
\end{equation}

To prove the theorem, we find a value of $b$ that satisfies both \eqref{b-condition-1} and \eqref{b-condition-2}. Set
\[
b = \left[ \frac{\pi^2 n}{4t} \log\left( \frac{|\G|}{\e} \right) \right]^{1/2},
\]
so that equality holds in \eqref{b-condition-2}. Then \eqref{b-condition-1} holds if and only if $t$ satisfies the lower bound in the second statement of Theorem \ref{smoothing-optimized}. This completes the proof.

For the first statement of Theorem \ref{smoothing-optimized}, we retain the same value of $b$. Given $h \in \HH$, choose $x \in \Delta^{-1} \Z^{n-1}$ such that $h(v_j) = e^{2\pi i x_j}$ and $|x_j| \leq 1/2$ for all $j$. Define
\begin{align*}
\HH_1 &= \{h \in \HH \st \max_{1 \leq j \leq n-1} 2\pi |x_j| \leq b \}, \\
\HH_2 &= \{h \in \HH \st \max_{1 \leq j \leq n-1} 2\pi |x_j| > b \}.
\end{align*}
Since
\[
1 - \Re(\lambda_h) = \frac{1}{n} \sum_{j=1}^{n-1} [1 - \cos(2\pi x_j)],
\]
for any $h \in \HH_1$ we can apply Lemma \ref{cos-upper-bound} to get
\[
1 - \Re(\lambda_h) \geq \frac{4\pi^2}{n} \left( \frac{1}{2} - \frac{b^2}{24} \right) \|x\|_2^2
\]
as in the previous argument. Likewise, for any $h \in \HH_2$ we have
\[
1 - \Re(\lambda_h) \geq \frac{1}{n}[1 - \cos(b)] \geq \frac{1}{n} \cdot \frac{2b^2}{\pi^2},
\]
keeping in mind that $\HH_2$ can only be nonempty when $b < \pi$.

With this preparation, we decompose
\[
\|H_{id}^t - U\|_2^2 = \sum_{h \in \HH_1 \setminus \{1\}} e^{-2t(1 - \Re(\lambda_h))} + \sum_{h \in \HH_2} e^{-2t(1 - \Re(\lambda_h))}.
\]
From here the proof is the same as before, substituting $\Delta \Z^{n-1}$ for $\overbar{\Delta} \Z^n$.
\end{proof}

The proof of Theorem \ref{cutoff-upper-bd} is elementary. No effort has been made to optimize the constants in the non-leading terms.

\begin{proof}[Proof of Theorem \ref{cutoff-upper-bd}]
The only graph on $3$ vertices with nontrivial sandpile group is the cycle $C_3$, for which the result can be checked directly. Therefore we may assume that $n \geq 4$.

Starting from the second statement of Theorem \ref{smoothing-optimized}, we use \eqref{smoothing-discrete-bound} along with the inequality $\log |\G| \leq (n-1)\log(d_\ast)$. For any $\e > 0$, if
\begin{equation}
\label{t-lower-bound-3}
t \geq \frac{1}{4\pi^2} (d_\ast^2 + d_\ast) n \log[2(n-1)(1 + 1/\e)] + \frac{\pi^2}{48} n[ (n-1)\log(d_\ast) + \log(1/\e)],
\end{equation}
then $\max\left\{ \|H_{id}^t - U\|_2^2, \, \|P_{id}^t - U\|_2^2 \right\} \leq 2\e$. Since $\log(1/\e) < d_\ast^2 \log(2 + 2/\e)$, the right side of \eqref{t-lower-bound-3} is less than
\begin{equation}
\label{t-lower-bound-4}
\frac{1}{4\pi^2} (d_\ast^2 + d_\ast) n [\log(n) + \log(2 + 2/\e)] + \frac{\pi^2}{48}[n^2 \log(d_\ast) + d_\ast^2 n \log(2 + 2/\e)].
\end{equation}
Using the bound $d_\ast \log(n) \leq n \log(d_\ast)$ (since $2 \leq d_\ast < n$ and $n \geq 4$) along with $d_\ast^2 + d_\ast < 2 d_\ast^2$, \eqref{t-lower-bound-4} is less than
\begin{equation}
\label{t-lower-bound-5}
\frac{1}{4\pi^2} d_\ast^2 n \log(n) + \left[ \frac{1}{4\pi^2} + \frac{\pi^2}{48} \right] n^2 \log(d_\ast) + \left[ \frac{2}{4\pi^2} + \frac{\pi^2}{48} \right] d_\ast^2 n \log(2 + 2/\e).
\end{equation}
Since
\[
\frac{1}{4\pi^2} + \frac{\pi^2}{48} < \frac{1}{4},
\]
any $t \geq \frac{1}{4\pi^2} d_\ast^2 n \log(n) + \frac{1}{4} n^2 \log(d_\ast) + c d_\ast^2 n$ will be greater than \eqref{t-lower-bound-5} as long as
\begin{equation}
\label{c-lower-bound}
c \geq \left( \frac{2}{4\pi^2} + \frac{\pi^2}{48} \right) \log(2 + 2/\e).
\end{equation}
It remains to show that \eqref{c-lower-bound} holds for $\e = \frac{1}{2} e^{-2c}$ whenever $c \geq 1$, as this implies that $\max\left\{ \|H_{id}^t - U\|_2, \, \|P_{id}^t - U\|_2 \right\} \leq \sqrt{2\e} = e^{-c}$. Using $e^{-2c} \leq 1$ followed by $c \geq 1$, we have
\[
\log(2 + 2/\e) = \log\left( \frac{e^{-2c} + 2}{\frac{1}{2} e^{-2c}} \right) \leq \log(6) + 2c \leq \left( \frac{2}{4\pi^2} + \frac{\pi^2}{48} \right)^{-1} c,
\]
and the proof is complete.
\end{proof}

\section{Examples}
\label{Examples}

We conclude by illustrating our results in a variety of specific examples. We will use the notation  $[n]=\{1,\ldots,n\}$.

\subsection*{Cycle graph}
\label{Cycle graph}

If $G=C_{n}$ is an $n$-cycle with successive vertices $v_{1},\ldots,v_{n-1},v_{n}=s$, then for each 
$n$th root of unity $\omega$, the function defined by $h_{\omega}(v_{k})=\omega^{k}$ is easily 
seen to be multiplicative harmonic. 
As there are $n$ spanning trees in an $n$-cycle, this completely accounts for the characters of 
$\G(C_{n})$. Moreover, this shows that $\G(C_n) \cong \HH(C_n) \cong \Z / n\Z$.

Since $\sum_{k=1}^{n}\omega^{k}=0$ whenever $\omega$ 
is a nontrivial $n$th root of unity, we see that $\lambda=0$ is an eigenvalue of the discrete time sandpile chain with multiplicity 
$n-1$. Thus \eqref{L2_equality} shows that the chain is stationary after a single step.

\subsection*{Complete graph}
\label{Complete graph}

Let $G=K_{n}$ be the complete graph on $n\geq 3$ vertices, $v_{1},\ldots,v_{n-1},v_{n}=s$. 
Set $M=\left\{ z\in\left(\Z/n\Z\right)^{n}:\sum_{j=1}^{n-1}z_{j}=z_{n}=0\right\}$ and $\omega=e^{2\pi i / n}$. 
For every $z\in M$, the function $h_{z}$ defined by $h_{z}(v_{j})=\omega^{z_{j}}$ is multiplicative harmonic since 
\[
\prod_{j=1}^{n}h_{z}(v_{j})=\omega^{\sum_{j=1}^{n}z_{j}}=1
\]
and thus 
\[
\prod_{j\neq k}h_{z}(v_{j})=h_{z}(v_{k})^{-1}=h_{z}(v_{k})^{n-1}
\]
for all $k=1,\ldots,n$. 
Every element of $M$ is uniquely determined by specifying the first $n-2$ coordinates, so 
$\left|M\right|=n^{n-2}$.
As this is equal to the number of spanning trees of $K_{n}$ by Cayley's formula \cite{Dieter}, 
we have $\HH=\{h_{z}\}_{z\in M}$, so the eigenvalues of the discrete time sandpile chain on $G$ are $\lambda_{z}=\frac{1}{n}\sum_{j=1}^{n}\omega^{z_{j}}$ as $z$ ranges over $M$. This characterization of $\HH$ shows that $K_n$ has sandpile group $(\Z / n\Z)^{n-2}$.

By construction, no $z\in M \setminus \{\zero\}$ can have all coordinates in $\{0,1\}$ or $\{0,-1\}$, so $z_{\ast}=(1,-1,0,\ldots,0)$ gives the maximum modulus of a nontrivial eigenvalue. 
(Any permutation of the first $n-1$ coordinates of $z_{\ast}$ or 
of $\pm (2,1,\ldots,1,0)$ also gives a nontrivial eigenvalue of maximum modulus.) The inequality $\cos(x)\leq 1 - \frac{27}{8\pi^2}x^{2}$ for $\left|x\right|\leq 2\pi/3$ (Lemma \ref{cos-upper-bound}) gives
\[
\lambda_{z_{\ast}} =\frac{1}{n}\left(n-2+\omega+\omega^{-1}\right) = 1-\frac{2}{n}\left(1-\cos\left(\frac{2\pi}{n}\right)\right) \leq 1-\frac{27}{n^{3}}.
\]
Using $\cos(x)\geq1-x^{2}/2$, we see that the relaxation time of the sandpile chain is 
	\[ \trel (K_n) = \Theta(n^{3}). \] 

Now the standard bounds relating relaxation and mixing times (Proposition \ref{TV_eigval}) imply the order of $\tmix(K_n)$ is between $n^{3}$ and $n^{4}\log n$.  Our next goal is to prove Theorem~\ref{t.cutoff.intro}, which says that the truth is in between: $\tmix(K_n)$ has order $n^3 \log n$, and moreover the sandpile chain on $K_n$ exhibits cutoff at time $\frac{1}{4\pi^{2}} n^3\log n$.
The upper bound follows easily from Theorem~\ref{cutoff-upper-bd}, which gives:

\begin{prop}
\label{K_n_upper}
Let $P$ be the transition operator for the discrete time sandpile chain on $K_{n}$ and $U$ the stationary distribution. 
For any $c \geq 5/4$, if $t\geq\frac{1}{4\pi^{2}}n^{3}\log n+cn^{3}$, then 
$\| P_{id}^{t}-U \|_{\TV} \leq e^{-c}$.
\end{prop}

To obtain a matching lower bound, we turn to the eigenfunctions. 
The basic idea is to find a \emph{distinguishing statistic} $\varphi$ for which the distance between 
the pushforward measures $P_{id}^{t}\circ\varphi^{-1}$ and $U\circ\varphi^{-1}$ can be bounded 
from below using moment methods. Specifically, we appeal to Proposition 7.8 in \cite{LPW}, which
shows that for any $\varphi:\G\rightarrow\R$, 
\begin{equation}
\label{lower_bound_ineq}
\| P_{id}^{t}-U\|_{\TV}\geq 1-\frac{4}{4+R(t)}\text{ whenever }
R(t)\leq\frac{2\left(\EE_{P_{id}^{t}}[\varphi]-\EE_{U}[\varphi]\right)^{2}}
{\Var_{P_{id}^{t}}(\varphi)+\Var_{U}(\varphi)}.
\end{equation}
Natural candidates for $\varphi$ are eigenfunctions of 
$P$ corresponding to a large real eigenvalue $\lambda$. The reason for using such eigenfunctions is 
that if $X_{t}$ is distributed as $P_{id}^{t}$ and $Y$ has the uniform distribution, then 
$\EE\left[\varphi\left(X_{t}\right)\right]=\lambda^{t}\varphi(id)$ and $\EE\left[\varphi\left(Y\right)\right]=0$, 
so their difference is relatively large when $t$ is not too big. When $\lambda$ has high multiplicity, we can average over a basis of eigenfunctions to reduce the variances of $\varphi\left(X_{t}\right)$ and $\varphi\left(Y\right)$.

\begin{prop}
\label{K_n_lower}
Let $P$ and $U$ be as in Proposition \ref{K_n_upper}.
For any $c \geq 0$, we have
\[
\| P_{id}^{t}-U\| _{\TV} \geq 1 - \frac{100}{100 + e^{4\pi^2 c}} \enspace \text{\em for all } 0 \leq t \leq \frac{1}{4\pi^{2}}n^{3}\log n - cn^{3}.
\]
\end{prop}

Taken together, Propositions \ref{K_n_upper} and \ref{K_n_lower} show that the discrete time sandpile chain on $K_{n}$ exhibits cutoff at time $\frac{1}{4\pi^{2}}n^{3}\log n$ with cutoff window of size $O(n^3)$. Also, Proposition \ref{K_n_lower} implies the lower bound in Theorem \ref{t.cutoff.intro} since for all $c \geq 5/4$,
\[
\frac{100}{100 + e^{4\pi^2 c}} \leq \frac{100}{e^{4\pi^2 c}} \leq e^{-35c}.
\]

\begin{proof}[Proof of Proposition \ref{K_n_lower}]
First note that the statement is true when $n = 3$ since the only integer $t$ in the allowed range is $t = 0$, where the inequality can be checked directly. Therefore we may assume that $n \geq 4$.

To describe our choice of $\varphi$, write $D_{n}=\left\{ (j,k)\in[n-1]^{2}:j\neq k\right\}$ and set $\omega=e^{2\pi i / n}$.
For $(j,k)\in D_{n}$, define
\[
h_{j,k}(v)=\begin{cases}
\omega, & v=v_{j}\\
\omega^{-1}, & v=v_{k}\\
1, & \text{else}
\end{cases}.
\]
It is clear that $h_{j,k}\in\HH$, so the function $f_{j,k}:\G\rightarrow\T$ given by 
\[
f_{j,k}(\eta)=\prod_{v\in V}h_{j,k}(v)^{\eta(v)}=\omega^{\eta(v_{j})-\eta(v_{k})}
\]
is an eigenfunction for the sandpile chain with eigenvalue 
\[
\lambda_{1}=\frac{1}{n}\sum_{v\in V}h(v)=1-\frac{1}{n}\left[2-2\cos\left(\frac{2\pi}{n}\right)\right].
\]
As $\lambda_{1}$ does not depend on $(j,k)$, the function
\[
\varphi(\eta) := \frac{1}{(n-1)(n-2)}\sum_{(j,k)\in D_{n}}f_{j,k}(\eta)
\]
is also in the $\lambda_{1}$-eigenspace. Moreover, $\varphi$ is $\R$-valued because 
$f_{j,k} = \overbar{f_{k,j}}$.

To apply \eqref{lower_bound_ineq}, we first observe that since $\varphi(id)=1$ and $P\varphi=\lambda_{1}\varphi$, 
\[
\EE_{P_{id}^{t}}[\varphi]=\lambda_{1}^{t}\text{ and }\EE_{U}[\varphi]=0.
\]
(The latter expectation used the fact that $U$ is a left eigenfunction with eigenvalue $1\neq \lambda_{1}$, 
so $\EE_{U}[\varphi]=\left\langle U,\varphi\right\rangle =0$.)
So the numerator of \eqref{lower_bound_ineq} is
\[
2\left( \EE_{P_{id}^{t}}[\varphi]-\EE_{U}[\varphi]\right)^{2} = 2\lambda_{1}^{2t}.
\]

Next we need an upper bound on the denominator $\Var_{P_{id}^{t}}(\varphi)+\Var_{U}(\varphi)$.
Since $\HH$ is closed under pointwise multiplication, we have $h_{j,k} h_{l,m} \in \HH$ 
for all pairs $(j,k),(l,m)\in D_{n}$. The corresponding eigenfunction is 
\[
f_{j,k,l,m}(\eta) :=\prod_{v\in V}h_{j,k}(v)^{\eta(v)}h_{l,m}(v)^{\eta(v)}=f_{j,k}(\eta)f_{l,m}(\eta).
\]
The associated eigenvalue depends on the $4$-tuple $(j,k,l,m)$. For example, when $(j,k)=(l,m)$, the product $h_{j,k} h_{l,m}$ sends $v_{j}$ to $\omega^{2}$ and
$v_{k}$ to $\omega^{-2}$ and all other vertices to $1$, giving an eigenvalue of 
$1-\frac{1}{n}\left[2-2\cos\left(\frac{4\pi}{n}\right)\right]$.
If $(k,j)=(l,m)$, then $h_{j,k} h_{l,m}\equiv 1$. If $j,k,l$ are distinct and $k=m$, then $h_{j,k} h_{l,m}$ sends $v_{j}$ and $v_{l}$ to $\omega$ and $v_{k}$ to $\omega^{-2}$ and all other vertices to $1$, so the resulting eigenvalue is 
$1-\frac{1}{n}\left[3-2\omega-\omega^{-2}\right]$. 

Writing
\[
(n-1)^{2}(n-2)^{2}\varphi^{2} = \sum_{(j,k),(l,m) \in D_n} f_{j,k,l,m}
\]
as a linear combination of eigenfunctions, we can count the number of $f_{j,k,l,m}$ corresponding to each eigenvalue. This information is given in Table~\ref{table.eigen}.

\begin{table}[b]
\begin{center}
\begin{tabular}{|M{5cm}|M{6cm}|N}
\hline 
Eigenvalue & Multiplicity in $(n-1)^2(n-2)^2 \varphi^{2}$ & \\[5pt]
\hline 
\hline
$\lambda_{0}=1$ & $(n-1)(n-2)$ & \\[5pt]
\hline 
$\lambda_{1}=1-\frac{1}{n}\left[2-2\cos\left(\frac{2\pi}{n}\right)\right]$ & $2(n-1)(n-2)(n-3)$ & \\[5pt]
\hline 
$\lambda_{2}=1-\frac{1}{n}\left[2-2\cos\left(\frac{4\pi}{n}\right)\right]$ & $(n-1)(n-2)$ & \\[5pt]
\hline 
$\lambda_{3}=1-\frac{1}{n}\left[4-4\cos\left(\frac{2\pi}{n}\right)\right]$ & $(n-1)(n-2)(n-3)(n-4)$ & \\[5pt]
\hline 
$\lambda_{4}=1-\frac{1}{n}\left[3-2\omega-\omega^{-2}\right]$ & $(n-1)(n-2)(n-3)$ & \\[5pt]
\hline 
$\overline{\lambda_{4}}=1-\frac{1}{n}\left[3-2\omega^{-1}-\omega^{2}\right]$ & $(n-1)(n-2)(n-3)$ 
& \\[5pt]
\hline 
\end{tabular}
\smallskip
\caption{Expansion of $(n-1)^2(n-2)^2 \varphi^2$ in the eigenbasis for $P$. \label{table.eigen}}
\end{center}
\end{table}

It follows that
\[
\begin{split}
\Var_{P_{id}^{t}}(\varphi)+\Var_{U}(\varphi) 
& =\EE_{P_{id}^{t}}\left[\varphi^{2}\right]
+\EE_{U}\left[\varphi^{2}\right]-\EE_{P_{id}^{t}}\left[\varphi\right]^{2}-\EE_{U}\left[\varphi\right]^{2}\\
 & =\frac{2}{(n-1)(n-2)}+\frac{2(n-3)}{(n-1)(n-2)}\lambda_{1}^{t}+\frac{1}{(n-1)(n-2)}\lambda_{2}^{t} \\
& \qquad +\frac{(n-3)(n-4)}{(n-1)(n-2)}\lambda_{3}^{t}
+\frac{2(n-3)}{(n-1)(n-2)}\Re \left(\lambda_{4}^{t}\right)-\lambda_{1}^{2t}.
\end{split}
\]

To simplify, we note that $\lambda_3 \leq \lambda_1^2$. Thus
\[
\frac{(n-3)(n-4)}{(n-1)(n-2)}\lambda_{3}^{t}-\lambda_{1}^{2t}\leq 0.
\]
In addition, $|\lambda_4| \leq \lambda_1$, so $\Re(\lambda_4^t) \leq |\lambda_4|^t \leq \lambda_1^t$. This follows from the computation
\[
\lambda_1^2 - |\lambda_4|^2 = \frac{2}{n} \left(1 - \frac{4}{n} \right) \left[ 1 - \cos\left( \frac{4\pi}{n} \right) \right] + \frac{4}{n^2} \left[ \cos\left( \frac{2\pi}{n} \right) - \cos\left( \frac{6\pi}{n} \right) \right],
\]
where all of the terms are nonnegative since $n \geq 4$. Finally, $\lambda_2 < \lambda_1$. Therefore
\[
\Var_{P_{id}^{t}}(\varphi)+\Var_{U}(\varphi) \leq \frac{2}{(n-1)(n-2)} + \frac{4n-11}{(n-1)(n-2)} \lambda_1^t \leq \frac{6}{n^2} + \frac{5}{n} \lambda_1^t,
\]
again using $n \geq 4$ in the final inequality. The function
\[
R(t) = \frac{2 \lambda_1^{2t}}{\frac{6}{n^2} + \frac{5}{n} \lambda_1^t}
\]
thus satisfies the condition in \eqref{lower_bound_ineq}.

Next we find a lower bound on $\lambda_1^t$. Since $\cos(x) \geq 1 - x^2/2$, we have $\lambda_1 \geq 1 - 4\pi^2 / n^3$. Therefore
\[
t \log (\lambda_1) \geq \left( \frac{1}{4\pi^2} n^3\log(n) - cn^3 \right) \log\left( 1 - \frac{4\pi^2}{n^3} \right).
\]
Since $n \geq 4$, $4\pi^2 / n^3 \in [0,\pi^2/16]$. Now we use that $\log(1-x) \geq -x - x^2$ for $0 \leq x \leq \pi^2/16$. This is because the function $f(x) = x + x^2 + \log(1-x)$ satisfies $f(0) = 0$, $f(\pi^2/16) > 0$, and $f'(x) = x(1-2x)/(1-x)$, implying that $f$ is increasing on $[0, 1/2]$ and decreasing on $[1/2, \pi^2/16]$. This yields
\[
\begin{split}
t \log (\lambda_1) &\geq \left( \frac{1}{4\pi^2} n^3\log(n) - cn^3 \right) \left( -\frac{4\pi^2}{n^3} - \frac{16\pi^4}{n^6} \right) \\
&\geq -\log(n) + 4\pi^2 c - \frac{4\pi^2}{n^3} \log(n).
\end{split}
\]
Thus
\[
n \lambda_1^t = e^{\log(n) + t \log(\lambda_1)} \geq e^{-4\pi^2 \log(n) / n^3} e^{4\pi^2 c} \geq \alpha e^{4\pi^2 c},
\]
where $\alpha = e^{-4\pi^2 \log(4)/4^3}$. It follows that
\[
R(t) = \frac{2 (n \lambda_1^t)^2}{6 + 5(n \lambda_1^t)} \geq \frac{2 (\alpha e^{4\pi^2 c})^2}{6 + 5(\alpha e^{4\pi^2 c})} \geq \frac{2\alpha^2 e^{8\pi^2 c}}{(6 + 5\alpha) e^{4\pi^2 c}} = \left( \frac{2\alpha^2}{6 + 5\alpha} \right) e^{4\pi^2 c},
\]
where the first inequality used that the function $x \mapsto 2x^2/(6+5x)$ is increasing for $x \geq 0$. Since $2\alpha^2 / (6 + 5\alpha) > 0.04$, we conclude from \eqref{lower_bound_ineq} that
\[
\|P_{id}^t - U\|_\TV \geq 1 - \frac{4}{4 + 0.04 e^{4\pi^2 c}} = 1 - \frac{100}{100 + e^{4\pi^2 c}}. \qedhere
\]
\end{proof}

\subsection*{Complete bipartite graph}
\label{Complete bipartite graph}

Suppose that $G=K_{m,n}$ is the complete bipartite graph having vertices $\{u_{1},\ldots,u_{m},v_{1},\ldots,v_{n}\}$ and edges $\{\{u_{j},v_{k}\}\}_{j\in[m],k\in[n]}$. 
In this case, there are essentially two possible choices for the sink, $u_{m}$ or $v_{n}$. 
We will assume the former as the other case then follows by interchanging the $u$'s and $v$'s 
(or appealing to Proposition \ref{sink_change}). 
Any $h\in\HH$ must satisfy $h(u_{m})=1$ and thus 
$1=h(u_{m})^{n}=\prod_{k=1}^{n}h(v_{k})=h(u_{j})^{n}$ for all $j=1,\ldots,m$. 
If we let $h(u_{j})$ be arbitrary $n$th roots of unity for $j=1,\ldots,m-1$, then, writing $\rho=\prod_{j=1}^{m}h(u_{j})$, we must 
have $h(v_{k})^{m}=\rho$ for each $k=1,\ldots,n$. Letting $h(v_{1}),\ldots,h(v_{n-1})$ be arbitrary 
$m$th roots of $\rho$, 
$h(v_{n})$ is then determined by the condition $1=\prod_{k=1}^{n}h(v_{k})$. 
Any such function is multiplicative harmonic by construction, and there are $n^{m-1}m^{n-1}$ 
such functions corresponding to the different choices of $n$th roots of $1$ for $h(u_{1}),\ldots,h(u_{m-1})$ 
and $m$th roots of $\rho$ for $h(v_{1}),\ldots,h(v_{n-1})$. 
Since this is the number of spanning trees in $K_{m,n}$ \cite{Dieter}, 
we have found all possible choices for $h$.

When $3\leq m\leq n$, we claim that the discrete time sandpile chain has relaxation time $\trel = \Theta(n^{3})$. 
(The $m=2$ case will be discussed shortly.)
To see that this is so, observe that $d_{\ast}= n$, so Theorem \ref{lazy_gap_bd} implies that $\gamma_{d}\geq 8 / [n^2(m+n)] \geq 4 / n^3$.

Conversely, the function 
\[
h(v)=\begin{cases}
e^{\frac{2\pi i}{n}}, & v=u_{1}\\
e^{-\frac{2\pi i}{n}}, & v=u_{2}\\
1, & \text{else}
\end{cases} 
\]
is multiplicative harmonic and the corresponding eigenvalue has modulus 
\[
\left|\lambda_{h}\right|=\frac{1}{m+n}\left[m+n-2+2\cos\left(\frac{2\pi}{n}\right)\right]
=1-\frac{2}{m+n}\left(1-\cos\left(\frac{2\pi}{n}\right)\right),
\]
hence 
\[
\gamma_{d}\leq\frac{2}{m+n}\left(1-\cos\left(\frac{2\pi}{n}\right)\right)
\leq\frac{4\pi^{2}}{(m+n)n^{2}}
\leq \frac{4\pi^{2}}{n^{3}}.
\]

\subsection*{Torus}
\label{Torus}

Let $G = \Z_m \times \Z_m$ be the $m$-by-$m$ discrete torus (square grid with periodic boundary). Since $G$ is $4$-regular, each eigenvalue $\beta$ of the full Laplacian $\overbar{\Delta}$ arises from an eigenvalue $\lambda$ of the transition matrix $K$ for simple random walk on $G$, with $\beta=4(1-\lambda)$. Since $K$ is a tensor product of transition matrices for random walks on cycles, the eigenvalues of $K$ are 
$\lambda_{j,k}=\frac{1}{2}\left[\cos\left(\frac{2\pi j}{m}\right)+\cos\left(\frac{2\pi k}{m}\right)\right]$ for $j,k \in [m]$, so the matrix-tree theorem gives 
\[
\left|\G\right|=\frac{1}{m^2}\prod_{(j,k)\in[m]^2\setminus (m,m)}\left(4-2\cos\left(\frac{2\pi j}{m}\right)-2\cos\left(\frac{2\pi k}{m}\right)\right).
\]
Interpreting the log of the product as a Riemann sum shows that
\[
\frac{1}{m^2}\log\left(\left|\G(\Z_{m}\times\Z_{m})\right|\right)\to \frac{4\beta(2)}{\pi}\approx 1.1662\quad \text{as }m\to\infty
\]
where $\beta(2)$ is the Catalan constant \cite{GlaWu}.

At this point, it is convenient to state the following simple bound for random walks on abelian groups.

\begin{lemma}
\label{abelian_lower_bound}
Suppose that $A$ is an abelian group of order $N$ which is generated by a set $S=\{s_{1},\ldots,s_{n}\}$, 
let $\pi$ be the uniform distribution on $A$, and let $Q_{a}^{t}$ be the distribution of the random walk 
driven by the uniform measure on $S$ after $t$ steps started at $a$.
For any initial state $a$ and any $\e > 0$, 
\[
\| Q_{a}^{t}-\pi\|_{\TV}\geq 1-\e \enspace
\text{\em whenever } t\leq\log_2(\e N) - n.
\]
\end{lemma}

\begin{proof}
Since $A$ is abelian, the random walk at time $t$ must be at an element of the form 
$as_{1}^{x_{1}}\cdots s_{n}^{x_{n}}$ where $x_{1},\ldots,x_{n}\in\N$ 
with $x_{1}+\cdots+x_{n}=t$. 
Letting $S_{t}$ denote the set of all elements of this form, 
we have $\left|S_{t}\right|\leq\binom{t+n-1}{t} \leq 2^{t+n-1}$.
Since $Q_{a}^{t}$ is supported on $S_{t}$,  
\[
\| Q_{a}^{t}-\pi\|_{\TV}
\geq\left|Q_{a}^{t}\left(S_{t}\right)-\pi\left(S_{t}\right)\right|
\geq1-\frac{2^{t+n-1}}{N}\geq 1-\e
\]
whenever $t\leq\log_2(\e N) - n$. 
\end{proof}

We can now estimate the mixing time for the sandpile chain on the torus.

\begin{prop}
\label{torus_bound}
Let $P$ be the transition operator for the discrete time sandpile chain on $\Z_{m}\times\Z_{m}$ and $U$ the 
stationary distribution. \\
For any $c\geq 2$,
\[
\| P_{id}^{t}-U\|_{\TV} \leq e^{-\frac{2}{5}c}
\enspace \text{\em for all } t\geq\frac{5}{2}m^{2}\log(m)+cm^{2}.
\]
There exists $m_0 > 0$ such that if $m \geq m_0$, then for any $c \geq 0$,
\[
\| P_{id}^{t}-U\|_{\TV}\geq 1- 2^{-c}
\enspace \text{\em for all } t\leq 0.68 m^{2} - c.
\]
\end{prop}

\begin{proof}
The upper bound follows from Theorem \ref{smoothing-mixing-time} and the lower bound from 
Lemma \ref{abelian_lower_bound} along with the fact that 
$\left|\G(\Z_{m}\times\Z_{m})\right|\geq e^{1.166m^2}$ for $m$ sufficiently large. 
\end{proof}

Note that the preceding estimate gives bounds which differ only by a log factor of the number of vertices, 
and it did not require computing any of the approximately $e^{1.166m^{2}}$ multiplicative 
harmonic functions.

\subsection*{Continuous time; Location of the sink}
\label{Discrete vs Continuous}

So far in this section, all our examples have been in discrete time. Recall from Section \ref{Discrete and continuous time} that the moduli of the eigenvalues (and thus the $L^2$ mixing time) of the discrete time sandpile chain do not depend on the location of the sink. 
However, changing the sink does affect the \emph{arguments} of the eigenvalues and thus the $L^2$ mixing 
time for the continuous time sandpile chain. The next two examples illustrate this dependence.

\begin{example}[Triangle with tail]
\label{Triangle with Tail}
Let $G = (V,E)$, where $V = \{u,v,w,w_1,\ldots,w_m\}$ and
\[
E = \{\{u,v\},\{u,w\},\{v,w\}, \{w,w_1\},\{w_1,w_2\},\ldots,\{w_{m-1},w_m\}\}.
\]
This graph is a triangle with a long `tail' attached and has $n=m+3$ vertices. 
Any multiplicative harmonic function $h$ on $G$ must satisfy $h(w) = h(w_1) = \cdots = h(w_m)$. 
This can be seen by starting from $w_m$ and working backwards.

Let $\HH_u$ be the group of multiplicative harmonic functions on $G$ with the sink placed at $u$. 
Since $G$ has three spanning trees, $|\HH_u| = 3$. For $z \in \{1, e^{2\pi i/3}, e^{4 \pi i/3} \}$, 
we define $h_z \in \HH_u$ by $h_z(u) = 1$, $h_z(w) = h_z(w_1) = \cdots = h_z(w_m) = z$, and 
$h_z(v) = z^2$. If $z = 1$ then $h_1 \equiv 1$ and the associated eigenvalue is $1$. 
If $z$ is one of the other cube roots of unity, the eigenvalue is $\lambda_{h_z} = (1 - \frac{3}{n})z$. 
Therefore the spectral gaps of the discrete and continuous time sandpile chains have different orders: 
$\gamma_d = \frac{3}{n}$ while 
$\gamma_c = \frac{3}{2}(1 - \frac{1}{n}) = \frac{n-1}{2} \gamma_d$. 
The $L^2$ distances from stationarity are also different:
\begin{align*}
\|P_{id}^t - U\|_2^2 &= 2 \left( 1 - \frac{3}{n} \right)^{2t}, \\
\|H_{id}^t - U\|_2^2 &\leq 2e^{-2t}.
\end{align*}

If the sink is placed at $v$, then the eigenvalues are the same as when the sink is at $u$. 
By contrast, if the sink is placed at $w$ or any of the $w_j$, then the multiplicative harmonic functions 
$\{h'_z : z^3 = 1\}$ are given by $h'_z(w) = h'_z(w_1) = \cdots = h'_z(w_m) = 1$, $h'_z(v) = z$, 
and $h'_z(u) = z^2$. The associated eigenvalues are $1$ and $1 - \frac{3}{n}$ with multiplicity $2$. 
Hence $\gamma_c = \gamma_d = \frac{3}{n}$.

For intuition on this example, suppose $\eta$ is a recurrent chip configuration on $G$. 
We take one step in the discrete time chain by adding a chip at a uniformly chosen vertex and then 
stabilizing. It can be shown that adding a chip at any of $w, w_1, \ldots, w_m$ leads to the same 
configuration $\eta'$, so $P(\eta, \eta') = \frac{n-2}{n}$. 
When the sink is at $u$ or $v$, the transition matrix can be written as
\[
P^{(u)} = P^{(v)} = \frac{1}{n} \begin{bmatrix} 1 & n-2 & 1 \\
1 & 1 & n-2 \\
n-2 & 1 & 1 \end{bmatrix},
\]
which is nearly periodic. 
When the sink is at $w$ or one of the $w_j$, the transition matrix is
\[
P^{(w)} = P^{(w_j)} = \frac{1}{n} \begin{bmatrix} n-2 & 1 & 1 \\
1 & n-2 & 1 \\
1 & 1 & n-2 \end{bmatrix}.
\]
Both discrete time chains converge to stationarity at the same rate, but the continuous time chain associated with $P^{(u)} = P^{(v)}$ converges much faster than the continuous time chain associated with $P^{(w)} = P^{(w_j)}$.
\end{example}

\begin{example}[Unbalanced complete bipartite graph]
\label{Unbalanced complete bipartite graph}
Let $G = K_{2,m}$ be the complete bipartite graph where one part has two vertices. 
Here $V = \{u_1,u_2, v_1,\ldots,v_m\}$ and $E = \{\{u_i,v_j\} : 1 \leq i \leq 2, 1 \leq j \leq m\}$. 
A full characterization of all the multiplicative harmonic functions was given previously.

When the sink is placed at $u_1$, the multiplicative harmonic function $h \in \HH_{u_1}$ given by 
$h(u_1) = 1$, $h(u_2) = e^{2\pi i (2/m)}$, and $h(v_j) = e^{2\pi i/m}$ for all $j$ maximizes both 
$|\lambda_h|$ and $\Re(\lambda_h)$ over all $h \not\equiv 1$. 
Writing $n = m+2$, we compute
\begin{align*}
\gamma_d &= 1 - |\lambda_h| = \frac{2}{n} \left[ 1 - \cos\left( \frac{2\pi}{n-2} \right) \right] 
\approx \frac{4 \pi^2}{n^3}, \\
\gamma_c &= 1 - \Re(\lambda_h) = \left[ 1 + \frac{2}{n} \cos\left( \frac{2\pi}{n-2} \right) \right] 
\left[ 1 - \cos\left( \frac{2\pi}{n-2} \right) \right] \approx \frac{2\pi^2}{n^2}, \\
\frac{\gamma_c}{\gamma_d} &= \frac{n}{2} + \cos\left( \frac{2\pi}{n-2} \right) \approx \frac{n}{2}.
\end{align*}
If instead the sink is placed at $v_1$, then the function $h' \in \HH_{v_1}$ given by $h'(u_1) = e^{-2\pi i/m}$, 
$h'(u_2) = e^{2\pi i/m}$, and $h'(v_j) = 1$ for all $j$ maximizes both $|\lambda_{h'}|$ and 
$\Re(\lambda_{h'})$ over all $h' \not\equiv 1$. 
In this case, $\lambda_{h'}$ is a positive real number and
\[
\gamma_d = \gamma_c = \frac{2}{n} \left[ 1 - \cos\left( \frac{2\pi}{n-2} \right) \right] 
\approx \frac{4 \pi^2}{n^3}.
\]
\end{example}

\subsection*{Rooted sums} 
\label{Rooted sums}

Let $G_{1}=(V_{1},E_{1}),\ldots,G_{k}=(V_{k},E_{k})$ be simple connected graphs with distinguished 
vertices $s_{1},\ldots,s_{k}$, and let $G=(V,E)$ be any simple connected graph with $V=\{v_{1},\ldots,v_{k}\}$. 
Construct a new graph $G'$ from the $G_{j}$'s by identifying each $s_{j}\in V_{j}$ with $v_{j}\in V$. 
That is, $G'=(V',E')$ with $V'=\bigcup_{j=1}^{k}V_{j}$ and $E'=\left(\bigcup_{j=1}^{k}E_{k}\right)\bigcup\left\{ \{s_{i},s_{j}\}:\{v_{i},v_{j}\}\in E\right\} $.
Let $\HH_{j}$ denote the group of multiplicative harmonic functions on $G_{j}$ with sink vertex $s_{j}$, 
$\HH$ the multiplicative harmonic functions on $G$ with sink vertex $v_{k}$, 
and $\HH'$ those on $G'$ with sink $s_{k}$. 

Define the map $\psi:\HH\times\HH_{1}\times\cdots\times\HH_{k}\to\HH'$ by $\psi\left(h,h_{1},\ldots,h_{k}\right)=h'$ where $h':V'\to\T$ is given by $h'(v)=h(v_{j})h_{j}(v)$ 
for $v\in V_{j}$. One easily checks that $\psi$ is an isomorphism: It is a well-defined injective 
homomorphism by routine verification, and it is surjective because every spanning tree in $G'$ is 
formed by choosing a spanning tree from each of $G,G_{1},\ldots,G_{k}$. 
Writing $n_{j}=\left|V_{j}\right|$, $n=\left|V'\right|=\sum_{j=1}^{k}n_{j}$, 
we see that the eigenvalue for the discrete time sandpile chain on $G'$ associated with $h'=\psi\left(h,h_{1},\ldots,h_{k}\right)$ is 
\[
\lambda_{h'} =\frac{1}{n} \sum_{v\in V'}h'(v)
=\frac{1}{n}\sum_{j=1}^{k}h(v_{j})\sum_{w\in V_{j}}h_{j}(w)
=\frac{1}{n}\sum_{j=1}^{k}n_{j}h(v_{j})\lambda_{h_{j}}.
\]

\subsection*{Is a large spectral gap compatible with a large sandpile group?}
\label{Is a large spectral gap compatible with a large sandpile group?}

Our last example is a graph whose sandpile chain mixes relatively quickly. It is an instance of the rooted sum construction above:
Let $G=P_{k}$ be a path on $k$ vertices, and for $j \in [k]$ let $G_{j}$ be a cycle on $m_{j}$ vertices,  
with $2<m_{1}\leq\cdots\leq m_{k}$. Since $\HH(P_{k})$ is trivial, the multiplicative harmonic functions 
on the rooted sum $G'$ are given by $h'(v)=h_{j}(v)$ for $v\in G_j$ with $h_{j}\in\HH(G_j)$.  The sandpile group of $G'$ is isomorphic to 
$(\Z / m_1 \Z) \times \cdots \times (\Z / m_k \Z)$,
and the largest nontrivial eigenvalue of the sandpile chain is
	\[ \lambda_{*}
	=1-\frac{m_{1}}{m_{1}+\cdots+m_{k}}, \]
corresponding to choosing any nontrivial multiplicative harmonic function $h_1$ on the smallest cycle $G_1$, and $h_{2}\equiv \cdots \equiv h_k \equiv 1$.

When $m_1= \cdots = m_k =m$ the rooted sum $G'$ has $mk$ vertices, sandpile group of order $m^{k}$, 
and sandpile spectral gap $\gamma_c = \gamma_d = 1/k$. Fixing $k$ gives a sequence of graphs, indexed by $m$, with gap 
uniformly bounded away from $0$.  

The same construction with $k = n^\alpha$ and $m = n^{1-\alpha}$ for fixed $0 < \alpha < 1$ gives a graph with $n$ vertices, sandpile group of order $e^{(1-\alpha)n^\alpha \log(n)}$, and sandpile spectral gap $1/n^\alpha \gg 1/ d_\ast^2 n$. We conclude with two questions.

\begin{enumerate}
\item[1.] Does there exist a graph sequence $G_n$ with $\gamma_d(G_n) > c >0$ such that the size of the sandpile group $|\G(G_n)|$ grows faster than any power of $n$?
\smallskip
\item[2.] Does there exist a graph sequence $G_n$ with $\log |\G(G_n)| = \Omega(n \log d_\ast)$ such that $\gamma_d(G_n) \gg 1/ d_\ast^2 n$?
\end{enumerate}

\section*{Acknowledgments}

We thank Spencer Backman, Shirshendu Ganguly, Alexander Holroyd, Yuval Peres and Farbod Shokrieh for inspiring conversations.

\bibliographystyle{plain}
\bibliography{sandpile}

\end{document}